\documentclass[12pt]{amsart}
\usepackage{a4wide}
\usepackage{psfrag}
\usepackage{graphicx}
\usepackage{verbatim}
\usepackage{amssymb}

\numberwithin{equation}{section}
\newtheorem{sublem}{Sublemma}
\newtheorem{lemma}{Lemma}[section]
\newtheorem{prop}[lemma]{Proposition}
\newtheorem{thm}{Theorem}

\theoremstyle{definition}

\newtheorem{defi}[lemma]{Definition}
\theoremstyle{remark}
\newtheorem{rem}{Remark}

\newcommand{\T}{\mathbb{T}}

\newcommand{\vf}{\varphi}

\newcommand{\vp}{\varphi}
\newcommand{\ve}{\varepsilon}

\newcommand{\Z}{\mathbb{Z}}

\newcommand{\R}{\mathbb{R}}
\newcommand{\dist}{\text{dist}}
\newcommand{\inn}{\text{int}}

\newcommand{\tvp}{\widetilde{\varphi}}
\newcommand{\tpsi}{\widetilde{\psi}}
\title[On some generalizations of skew-shifts on $\mathbb{T}^2$]
{On some generalizations of skew-shifts on $\mathbb{T}^2$}
\author{Kristian Bjerkl\"ov}
\email{bjerklov@kth.se}

\address{Department of Mathematics, KTH Royal Institute of Technology, 100 44 Stockholm, Sweden}
\thanks{Research supported by the Swedish Research Council,  2012-3090.} 

\begin{document}
\begin{abstract}
In this paper we investigate maps of the two-torus $\T^2$ of the form $T(x,y)=(x+\omega,g(x)+f(y))$ for Diophantine $\omega\in\T$ and 
for a class of maps $f,g:\T\to\T$, where each $g$ is strictly monotone and of degree 2, and each $f$ is an orientation preserving 
circle homeomorphism. For our class of $f$ and $g$ we show that $T$ is minimal and has exactly two invariant and ergodic Borel probability measures.
Moreover, these measures are supported on two $T$-invariant graphs. One of the graphs is a Strange Nonchaotic Attractor whose basin of attraction
consists of (Lebesgue) almost all points in $\T^2$. Only a low regularity assumption (Lipschitz) is needed on the maps $f$ and $g$, and  
the results are robust with respect to Lipschitz-small perturbations of $f$ and $g$.

\bigskip
MSC2010: 37C40, 37C70, 37E30 
\end{abstract}
\maketitle
\begin{section}{Introduction}
Skew-shifts are classical examples in ergodic theory\footnote{They are also used, for example, as base dynamics in the theory of ergodic Schr\"odinger operators. 
See, e.g., \cite{BGS}.}.  These are maps of the two-torus $\T^2=\R^2/\Z^2$ of the form
$$
T_0(x,y)=(x+\omega,kx+y)
$$
where $k\in\mathbb{Z}\setminus \{0\}$ and $\omega\in\R\setminus \mathbb{Q}$. It is well-known that they are minimal and uniquely ergodic  
(see, e.g., \cite[Prop. 4.7.4]{BS}). A natural question is to ask what can happen with the dynamics of $T_0$ if one changes the 
fiber map $y\mapsto kx+y$ to, e.g.,  
$y\mapsto kx+f(y)$, where $f$ is a circle homeomorphism; 
this is the topic of the present paper. In particular we will be interested in the case $k=2$ which is the simplest case where one gets
"interacting critical sets", or "resonances"\footnote{This is due to the fact every circle $\T\times \{y\}$, $y\in\T$, is mapped to a closed curve which intersects
each circle $\T\times\{\eta\}$ twice.}. This case also naturally arises in certain models (see Example 2 below). 
Showing how such cases can be handled for certain classes of $f$, without using any parameter exclusion, 
and under low regularity assumptions, is one of the main purposes of this paper. 
The case $k=1$ is easier (but far from understood for general $f$) and becomes a special case of our analysis 
(for the class of maps $f$ under consideration; see Subsection \ref{assumptions}).
Other existing techniques (for example, \cite{B2, Y2}) could also have been adapted for analyzing the case $k=1$, 
since there is no need of any exclusion of parameters in this situation (as already noticed in \cite[Remark 2]{Y2}).

Skew-shifts are examples of so-called quasi-periodically driven circle maps (skew-product maps), i.e., maps on $\T^2$ of the form
$$
F(x,y)= (x+\omega,h(x,y))
$$
where each map $\T\ni y\mapsto h_x(y)=h(x,y)\in \T$ is a circle homeomorphism. This (large) class of maps (which all have zero topological entropy; see \cite{Bo}) 
have been intensely investigated in the literature
(see, e.g., \cite{AK,B1,B2,H2,Ja,KKHO,WZ,Y2}, and references therein), under various assumption on $h(x,y)$; 
but still a general theory seems very remote\footnote{The situation where the base dynamics $x\mapsto x+\omega$ is an ergodic translation on $\T^d$, $d\geq 2$,
still lacks of techniques for rigorous analysis in many interesting cases, including our setting. 
Furthermore, the case of continuous maps on $\T^2$, or $\T\times I$ ($I$ an interval), of the form
$F(x,y)= (x+\omega,q(x,y))$ 
where the maps $q(x,\cdot)$ are not necessarily invertible is much less understood; see, e.g., \cite{B3} for some results in this direction. }. 
One of the central problems that arises when analyzing such maps is the presence of "resonances". They appear in one way or another 
(and on infinitely many scales), for examples
in the form of so-called small divisors. There are essentially two ways to handle this problem: either to face it, or try to avoid it. A way to avoid
resonances is to consider parameter families of maps $F_a$, and use some kind of parameter exclusion to end up with a "large" set of parameter values 
$a$ for which one can describe the dynamics of $F_a$. This method has been used, e.g., in \cite{DS}. It is also the route taken in, for example, \cite{B1, Ja, Ja2, Y2} 
(based on ideas from one-dimensional dynamics, e.g., \cite{BC, CE, Jak}). The price one typically has to pay to be able to treat a fixed map $F$ 
(from some class of maps), or to be able to say something about the dynamics of $F_a$ for \emph{all} parameter values $a$, is to put more 
assumptions on the map $h(x,y)$ (and often work harder). Examples of the latter case are \cite{B2, H2, SS, WZ}. In this paper we shall put a topological condition
on $h(x,y)$ (as, e.g., in \cite{AK, Y1}). 

In the case of skew-shifts we have $h(x,y)=kx+y$. 
Note that the map $g(x)=kx$ has degree $k\neq 0$, and thus $T_0$ is not homotopic to the identity map. 
A big difference between maps $F$ that are homotopic to the identity, and those which are not, is that
the latter ones cannot have continuous invariant curves, i.e., there are no continuous functions 
$\varphi:\T\to\T$ such that $F(x,\varphi(x))=(x+\omega,\varphi(x+\omega))$ for all $x\in\T$ 
(see, e.g., \cite[Prop.4.2]{H2}). However, there can be (highly discontinuous) \emph{measurable} functions $\varphi:\T\to\T$
whose graphs are almost everywhere $F$-invariant. If such a "strange" invariant graph (which can, and do, also arise in maps $F$ homotopic to the identity)
attracts points, it is often called a Strange Nonchaotic Attractor (SNA). 
Such objects has shown to exist in various quasi-periodically driven models (see, e.g., \cite{B1,B2, H2, JTNO, Ja, Y2}, and references therein).  
There is also a huge literature focusing on SNAs based on numerical experiments.  

A way to measure contraction (note that there can only be contraction in the $y$-direction, since we have rotation in the $x$-direction) is via the 
(fibered) Lyapunov exponents. Given a point $(x,y)\in \T^2$ we use the notation
$$
(x_n,y_n)=F^n(x,y), \quad n\in\mathbb{Z}.
$$
Assuming, for example, that $h:\T^2\to\T$ is $C^1$, we can define the (fibered) Lyapunov exponents
$$
\lambda(x,y)=\limsup_{n\to\infty}\frac{1}{n}\log\left|\frac{\partial y_n}{\partial y}\right|
= \limsup_{n\to\infty}\frac{1}{n}\sum_{k=0}^{n-1}\log \left|\frac{\partial h}{\partial y}(x_k,y_k)\right|.
$$
For the map $T_0$ we trivially have $\lambda(x,y)=0$ for all $(x,y)\in\T^2$.
Furthermore, since the map $F$ is invertible, we can have $\lambda(x,y)>0$ for, at most, a set of zero Lebesgue measure of $(x,y)\in\T^2$. Thus, the question
is whether $\lambda(x,y)=0$ for a.e. $(x,y)\in\T^2$, or if $\lambda(x,y)<0$ for a positive measure set of $(x,y)\in\T^2$. 
Unfortunately this question is often notoriously hard to answer, but central for the description of the dynamics.

\subsubsection{Perturbations of skew-shifts}
We return to the skew-shift $T_0$. It has been noticed that the dynamics of this map seems to be sensitive to perturbations (in the fiber), 
but that the perturbed maps have a more robust dynamics. More precisely, in \cite{KKHO} the authors investigates maps of $\T^2$ of the form
$$
F(x,y)=(x+\omega,x+y+\eta P(x,y))
$$ 
where $P:\T^2\to\R$ is $C^1$, and the parameter $\eta$ ranges in an interval for which $F$ is invertible. 
An example considered in the numerical experiments in  \cite{KKHO} is
$P(x,y)=\sin(2\pi y)$. These experiments indicate that the map $F$  has an attractor with a negative (fibered) Lyapunov exponent for all (small) $\eta\neq 0$. 
Whether this is true or not seems to be an open question.
However, under such an assumption, it is rigorously shown in \cite{KKHO} that the map $F$ has a single attractor, and $F$ is topologically transitive. Thus,
assuming that the Lyapunov exponent is negative for an open set of parameters $\eta$ one gets robust strange nonchaotic attractors, meaning that the SNA survives 
after perturbation (in many other models it is expected that the SNA is very sensitive to perturbations, existing only on the boundary between order and chaos).
See Figure \ref{orbit} for a plot of an orbit of a similar kind of map.

In this paper we will show that such an situation, with robust SNAs, indeed can occur in models included in the wide class considered in \cite{KKHO}. See Theorem 1
below.  

In fact, robust families of SNAs also follow by considering the projective action of the non-uniformly hyperbolic $SL(2,\mathbb{R})$-cocycles 
(see next paragraph for more details on quasi-periodic cocycles) in \cite[Corollary 1]{Y2}, combined with \cite[Proposition 4.17]{H2}. 
See also Example 2 below for an extension of Young's result \cite[Corollary 1]{Y2}.

\begin{figure}
\includegraphics[width=7cm]{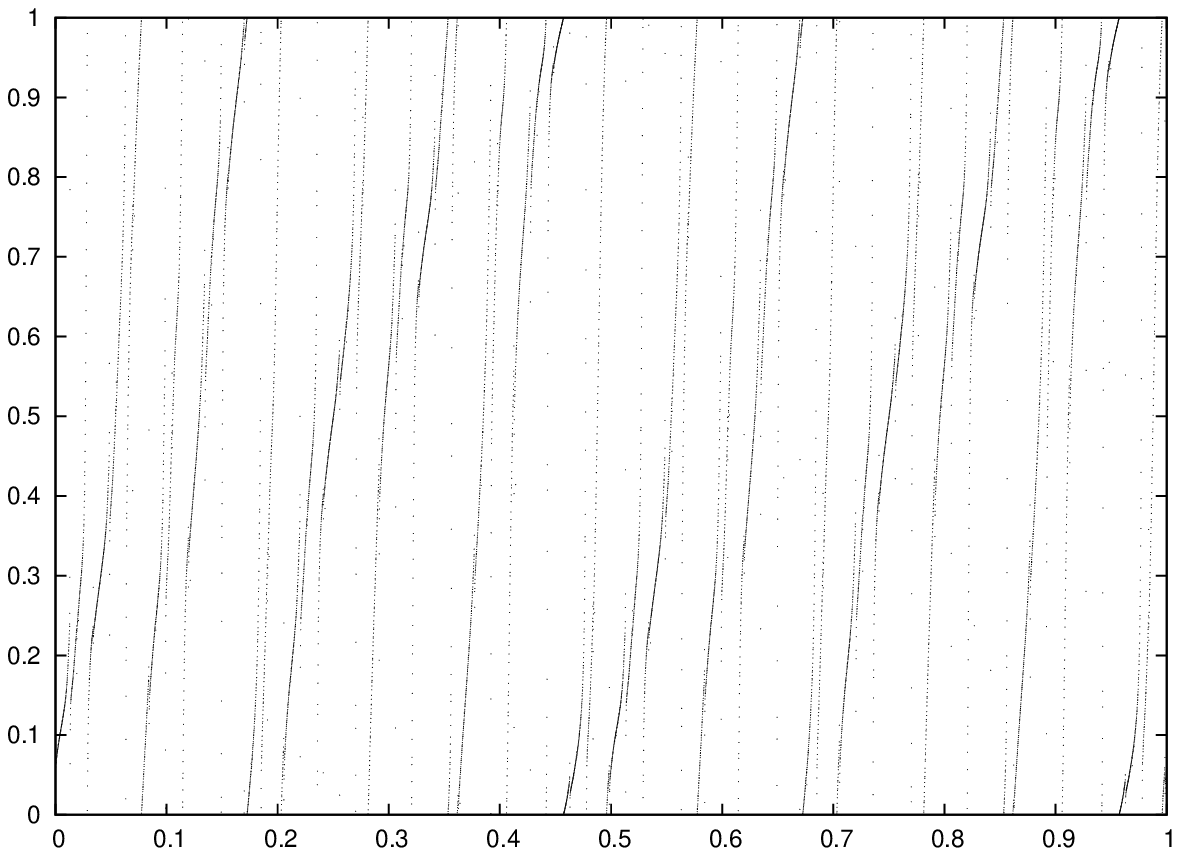}
\caption{Plot of (a piece) of the orbit of $(0,0)$ under the map $F(x,y)=\left(x+\frac{1}{\sqrt{2}},2x+y+0.1\sin2\pi y\right)$.}\label{orbit}
\end{figure}

\subsubsection{Cocycles}
A special class of maps $F$, and a well-studied one, arises in quasi-periodic $SL(2,\mathbb{R})$ cocycles. Consider maps $\Phi:\T\times \R^2\to\T\times \R^2$ 
$$
\Phi(x,v)=(x+\omega,M(x)v)
$$
where $M:\T\to SL(2,\R)$ is continuous. 
The cocycle $\Phi$ induces a projective action, denoted $\overline{F}$, on $\T\times \mathbb{P}^1$. If $v\in\R^2\setminus \{0\}$, 
we denote by $\overline{v}\in \mathbb{P}^1$ 
the projectivization of $v$. We introduce coordinates on $\mathbb{P}^1=\T$ by letting $\overline{\binom{\cos\pi y}{\sin \pi y}}$ correspond to $y\in \T$. 
Moreover, an invertible matrix $M$ induces a map $\overline{M}$ on $\mathbb{P}^1$ in the natural way.
Seen like this, the cocycle $\Phi$ induces a quasi-periodically driven circle map  $\overline{F}$.
For example, if $M(x)=R_{2\pi x}$, where 
\begin{equation}\label{Rtheta}
R_\theta=\left(\begin{matrix}\cos \theta & -\sin \theta \\ \sin \theta & \cos \theta\end{matrix}\right),
\end{equation}
then it is easy to check that $\overline{F}(x,y)=(x+\omega, 2x+y)$, i.e., we get $T_0$ with $k=2$.

To the cocycle $\Phi$ there is associated the (maximal) Lyapunov exponent
$$
L(M)=\lim_{n\to\infty}\frac{1}{n}\int_{\T}\log \|M^n(x)\|dx \geq 0
$$
where $M^n(x)=M(x+(n-1)\omega)\cdots M(x), n\geq 1$.
From Oseledets' theorem, it follows that if $L(M)>0$, then there exists a measurable splitting $W^+(x)\oplus W^-(x)=\R^2$ such that
$$
\lim_{n\to\infty}\frac{1}{n}\log|M^n(x)v|=L(M) \text{ for a.e. }x\in \T, \text{ all } v\notin W^-(x). 
$$
We note that if we let $v_n=M^n(x)v$ and  $w_n=M^n(x)w$, where $v,w\notin W^-(x)$ and $\overline{v}\neq \overline{w}$, then, since $M^n(x)\in SL(2,\R)$, we
have (by considering area)
$$
|v_n||w_n|\sin\theta_n=|v||w|\sin\theta
$$
where $\theta_n\in [0,\pi)$ is the angle between $v_n$ and $w_n$. For a.e. $x\in\T$ we then have
$$
\lim_{n\to\infty} \frac{1}{n}\log\sin\theta_n=-2L(M).
$$
Thus, we obtain the Lyapunov exponent $L(M)$ in the contraction rate (in the fiber) in iterates of $\overline{F}$. 

The by far most studied case is the one-parameter family 
$$
M(x)=M_E(x)=\left(\begin{matrix}0 & 1 \\ -1 & q(x)-E\end{matrix}\right), \quad E\in \R,
$$
where $q:\T\to\R$. This family of cocycles arises when investigating the one-frequency quasi-periodic Schr\"odinger equation 
$-(u_{n+1}+u_{n-1})+q(x+n\omega)u_n=Eu_n$. 
In this case the cocycles are homotopic to the identity. For real-analytic $q$ we are beginning to get a good understanding of many properties of 
the cocycles (see, for example, \cite{A} and references therein)  .
On the contrary, if $q$ is only assumed to have finite smoothness, then the results are much more sparse. If one uses parameter exclusion (in $E$) there are results in, e.g., \cite{B2}.
If we let $q(x)=K q_0(x)$ where $q_0:\T\to \R$ is $C^2$ and has exactly two
non-degenerate critical points ("cosine-like"), then it is possible to describe much of the dynamics of $\Phi_E$ (or $\overline{F}_E$) for all values of $E$, 
provided that the constant $K$ is large and $\omega$ is Diophantine 
\cite{B1, WZ}  (spectral properties of the associated Schr\"odinger operator, for this setting, is treated in \cite{FSW,S}). For example, one only gets SNA 
(for the associated projective map $\overline{F}$) for parameters $E$ in a Cantor set (of large measure) \cite{P,S,WZ2}; thus they are not robust.
However, for a fixed $q_0$ of finite smoothness with more than two monotonicity intervals it is not known what can happen; one would have to improve the techniques in the proofs
in \cite{B1, WZ}, since one in this case can have "multiple resonances" (although this is not "typical").   

Recently there has been an increased interest in quasi-periodic  
cocycles non-homotopic to the identity. In \cite{AK} it is shown that the (smooth) non-homotopic case (or, more precisely, the so-called monotonic case) 
is surprisingly well-behaved and much more robust than, 
for example, the Schr\"odinger
case. The results in the present paper is also an indication of this (the examples in \cite{Y1} are also related to this), that "monotonic" maps are more robust  
(for example, we can work with low regularity assumptions, even though we have resonances). In \cite{AK} it is also shown that the projective action of so-called 
pre-monotonic cocycles is always minimal. Only $C^{1+\ve}$ smoothness is needed for the argument (which is partly based on an argument from \cite{KKHO}).
An interesting problem is whether this holds under even lower regularity assumptions (see also \cite[Problem 3.1]{AK}). 

\subsection{Main object}
In this paper we shall consider maps  $T:\T^2\to\T^2$ of the form\footnote{We could have considered more general classes of maps, i.e., that 
we have the exact form $g(x)+f(y)$ is not important for the method we use; but for simplicity we have chosen to work with this class.}
\begin{equation}\label{MapT}
T(x,y)=(x+\omega,g(x)+f(y))
\end{equation}
where the functions $f,g:\T\to\T$ are assumed to be Lipschitz continuous.

We recall that a function $\varphi:\T\to\T$ is Lipschitz
continuous if there exists a constant $K$ such that $d(\vf(y),\vf(x))\leq Kd(x,y)$ for all $x,y\in \T$; here $d$ denotes the standard metric on $\T$.
The constant $K$ is called a Lipschitz constant for $\vf$. We say that $\vf$ has a Lipschitz constant $c$ on $I\subset \T$ if
$d(\vf(y),\vf(x))\leq c d(x,y)$ for all $x,y\in I$.
Recall that a Lipschitz continuous function is absolutely continuous, and thus differentiable a.e. on $\T$.

We also recall that if $\varphi:\T\to\T$ is a continuous function, then a 
lift of $\vf$ is a continuous function $\Phi:\R\to\R$ such that $\pi\circ\Phi(x)=f\circ \pi(x)$ for all $x\in \R$, where
$\pi:\mathbb{R}\to \T$ is the natural projection $\pi(x)=x \mod 1$. It is well-know that a lift of $\vf$ is unique up to 
the addition of an integer. Moreover, $\Phi(x+1)-\Phi(x)=k$, $k\in\Z$. The integer $k$ is called the degree of $\varphi$.

Note that for maps $T$ of the form (\ref{MapT}), the (fibered) Lyapunov exponents $\lambda(x,y)$ become
$$
\lambda(x,y)=\limsup_{n\to\infty}\frac{1}{n}\sum_{k=0}^{n-1}\log \left|f'(y_k)\right|.
$$

\subsection{Assumptions on $f,g$ and $\omega$}\label{assumptions}
Here we specify the exact assumptions on $f$ and $g$. We begin with the assumption on $g$:

\smallskip
\begin{itemize}
\item[$(A1)(\kappa)$]\emph{We assume that $g:\T\to\T$ has degree 2, and that a lift $G:\R\to\R$ of $g$ is bi-Lipschitz, and $G$ and $G^{-1}$ 
both have a Lipschitz constant $<\kappa$.}  
\end{itemize}

\smallskip

\begin{rem}An natural example is $g(x)=2x$. Our result (Theorem 1) also holds in the easier case when $g$ has degree $1$, for example if  $g(x)=x$. 
The proof is easier (we do not get any "resonant cases"; see Section 2.3, and Sections 5 and 6). 
However, one of the main purposes of this paper is to treat "resonances", so we do not stress on this case.

In our analysis it would in fact be possible to allow $g$ to have any degree $k\geq 1$ (and not adding any smoothness assumptions). 
However, the bookkeeping (and notation) in the proof of Theorem 1 below would become slightly more involved; but fundamentally no new problems would arise.   
Thus we have chosen to concentrate on the case $k=2$.
\end{rem}

In the assumptions on $f$ below, the number $\ve>0$ should be thought of being a (very) small number.
\begin{itemize}
\item[$(A2)(\ve,\rho)$]
\emph{Assume that $f:\T\to\T$ is an orientation preserving bi-Lipschitz homeomorphism, where $f$ and $f^{-1}$ both have a Lipschitz constant $\displaystyle <\ve^{-\rho}$. 
We further assume that there exist disjoint open intervals $A, B \subset \T$ such that
$$
|\T\setminus B|<\ve, |f(\T\setminus A)|<\ve 
$$
and $f$ has a Lipschitz constant $<\ve$ on $B$, and $f^{-1}$ has a Lipschitz constant $<\ve$ on $f(A)$.} 
\end{itemize}

\smallskip
That $f$ is bi-Lipschitz implies that $f$ and $f^{-1}$ both are differentiable for a.e. $y\in \T$.  We note that the two last conditions imply
$$
f'(y)<\ve \text{ for a.e } y\in B; \quad
f'(y)>\ve^{-1}  \text{ for a.e } y\in A.
$$
\begin{rem}
We have tried to make the assumptions on $f$ as clean as possible, avoiding the introduction of many parameters. It would be possible to have
other assumptions. However, we do really need that the assumptions are "perturbative", in the sense that the  $\ve$ is very small.
This will enable us, in the proof, to start the induction immediately. Without such a smallness assumption one would need to develop some
finer analysis on the base case, which should give estimates (similar to the ones we have put on $f$) for some iterate $T^k$. 
Such an approach would be very interesting. However, here we focus on one part of the main mechanisms. 
\end{rem}

\noindent {\bf Example 1.} An example of a map satisfying $(A2)(\ve,2)$ is
$$
f(y)=\begin{cases}2y/\ve, &0\leq y\leq a \\
\ve(y-1)/2+1, & a\leq y\leq 1
\end{cases}
$$
where $a=(2\ve-\ve^2)/(4-\ve^2)$. Here we can take $A=(0,a)$ and $B=(a,1)$. See Figure \ref{fig0}
\begin{figure}
\psfrag{x}{$y$}
\psfrag{y}{$z$}
\psfrag{f}{$z=f(y)$}
\psfrag{A}{$A$}
\psfrag{B}{$B$}
\psfrag{R}{$f(\T\setminus A)$}
\psfrag{1}{$1$}
\includegraphics[width=7cm]{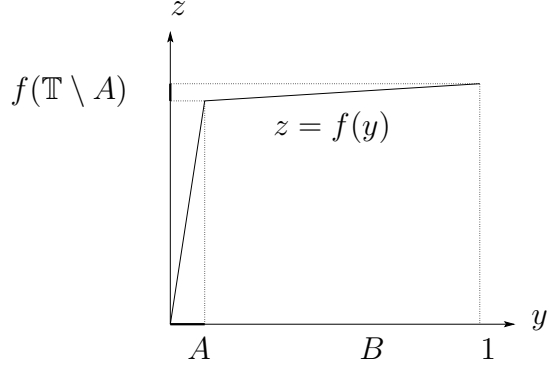}
\caption{The function $f$ in example 1.}\label{fig0}
\end{figure}

\begin{rem}Note that the assumptions on $f$ and $g$ are robust in the following sense: If $g$ satisfy $(A1)(\kappa)$ and $f$ satisfy $(A2)(\ve,\rho)$ (for some $\ve,\kappa,\rho>0$), and if 
$f_1,g_1:\T\to\T$ are Lipschitz continuous with a Lipschitz constant sufficiently small, then $g+g_1$ satisfy $(A1)(\kappa)$ 
and $f+f_1$ satisfy $(A2)(\ve,\rho)$. Note also that if $g$ satisfy $(A1)(\kappa)$, then, trivially, $g(x)+t$ also satisfy $(A1)(\kappa)$ for any
$t\in \T$.
\end{rem}

Next follows our assumption on the base frequency $\omega$. We assume that $\omega$ satisfies the Diophantine condition
\footnote{We could equally well have taken the lower bound of the form 
$\gamma/|q|^{1+\tau}$, any $\tau \geq 1$; but in order not to have many constants we have chosen to work with $q^2$. In fact one could even have a weaker condition,
like Brjuno, as, for example, in \cite{Y2}.}
\begin{equation}\label{DC}
\inf_{p\in\mathbb{Z}}|q\omega-p|>\frac{\gamma}{q^2} \quad \text{ for all } q\in\mathbb{Z}\setminus \{0\},  
\end{equation}
for some constant $\gamma>0$. Notice that this condition is satisfied for (Lebesgue) a.e. $\omega$ .

\subsection{Statement of results} We are now ready to state our results for maps $T$ defined as in (\ref{MapT}):

\begin{thm} Assume that $\omega$ satisfies the Diophantine condition (\ref{DC}) for some $\gamma>0$. Given any numbers $\kappa,\rho>1$ there is  
an $\ve_0=\ve_0(\gamma,\kappa,\rho)>0$ such that for all $0<\ve\leq \ve_0$ the following holds for the map $T(x,y)=(x+\omega,g(x)+f(y))$ whenever
 $g$ satisfies $(A1)(\kappa)$ and $f$ satisfies $(A2)(\ve,\rho)$: 
\begin{enumerate}
\item The map $T$ is minimal; 
\item The Lyapunov exponent $\lambda(x,y)<(\log\ve)/2$ for a.e. $(x,y)\in \T^2$; 
\item $T$ has exactly two invariant and ergodic 
Borel probability measures $\mu^s,\mu^u$.
Moreover,  $\mu^s,\mu^u$ are the push-forward of the Lebesgue measure on $\T$ by two maps $x\mapsto (x,u(x))$ and $x\mapsto (x,s(x))$, respectively, where
$u,s:\T\to\T$ are distinct measurable functions whose graphs are a.e. $T$-invariant.  
\item For a.e. $x\in \T$ and all $y\neq s(x)$ we have $d(y_n,u(x_n))<C(x,y)\ve^{n/2}, ~n\geq 0$.
\end{enumerate}

\end{thm}
\begin{rem}a) Condition (4) tells us that most points $(x,y)$ are attracted, under forward iteration by $T$, to the graph of the measurable function $\xi\mapsto u(\xi)$.

b) We note that if $f$ is also $C^1$, then, since we by Birkhoff's ergodic theorem have
$$
\lim_{n\to\infty}\frac{1}{n}\sum_{k=0}^{n-1}\log |f'(y_k))|=\int_{\T^2}\log|f'|d\mu^u \text{ for $\mu^u$-a.e. } (x,y)\in\T^2,
$$ 
it follows from (3) and (4) that 
$$
\lambda(x,y)=\lim_{n\to\infty}\frac{1}{n}\sum_{k=0}^{n-1}\log |f'(y_k)|=\int_\T\log|f'(u(x))|dx \text{ for a.e. } (x,y)\in \T^2.
$$ 

\end{rem}

\noindent {\bf Example 2.} We can apply the above theorem to a class of $SL(2,\mathbb{R})$-cocycles over Diophantine rotation. 
Consider the cocycle
$$
\Phi(x,v)=\left(x+\omega, R_{2\pi\vf(x)}Dv\right)
$$
where $\omega$ satisfies (\ref{DC}) for some $\gamma>0$, $D$ is the diagonal matrix
$$
D=\left(\begin{matrix} K & 0 \\ 0 & 1/K\end{matrix}\right), \quad K>1,
$$
$R_\theta$ is the rotation matrix defined in (\ref{Rtheta}), and $\varphi:\T\to\T$ is any bi-Lipschitz map. Let $L(K)$ denote the (maximal) Lyapunov exponent 
of the cocycle (recall the discussion in Subsection 1.0.2 above).
 
It is easy to check that the projectivization of $\Phi$ is 
$$\overline{F}(x,y)=(x+\omega, 2\varphi(x)+f(y))$$ where $f(y)=\arctan(\tan(\pi y)/K^2)/\pi$ (mod 1). Note that $g(x):=2\varphi(x)$ has degree 2, and let
$\kappa>1$ be a Lipschitz constant for $G$ and $G^{-1}$, where $G$ is a lift of $g$. Thus $g$ satisfy condition $(A1)(\kappa)$.
Furthermore, given a small $\ve>0$ it is easy to verify that $f(y)$ satisfies $(A2)(\ve,2)$ (i.e., $\rho=2$) provided that the constant $K$ is sufficiently large.
Hence, we can apply Theorem 1 (with the given $\gamma, \kappa$ and $\rho=2$) and conclude that the projective action $\overline{F}$ is minimal
and that the (fibered) Lyapunov exponent $\lambda(x,y)<\text{const.} \log(1/K)$, and thus $L(K)>\text{const.} \log K$, for all sufficiently large $K$ 
(depending on $\gamma$ and $\kappa$).

\subsection{Outline of the paper} The rest of the paper is organized as follows: In Section 2 we introduce notations and define some important sets related to the map $T$.
We also derive basic properties of iterations of $T$. In the last part of Section 2 we derive a few elementary topological properties
which we shall use to control the geometry and topology of iterations of curves. Thus, they are central for our analysis. 
In Section 3 we begin by briefly discussing properties of Lipschitz continuous
functions. We also derive formulae for the derivative of the iterates, and use them to obtain quantitative estimates under suitable assumptions on the iterates. 
Theorem 1 is proved in Section 4, assuming that the main proposition, Proposition \ref{pt_prop}, holds. The latter proposition gives a very detailed description
of the dynamics of $T$ (under the given assumptions). In the following two sections, Section 5 and 6, Proposition \ref{pt_prop} 
is proved via an inductive scheme. These two sections contain the key part of the analysis. In Section 5 we verify the base case of the construction.
The general approach for analyzing the map $T$ is similar to the 
ones we have used in \cite{B1,B2} (see also \cite{Ja, Ja2} where the methods in \cite{B2} are adapted to larger classes of circle maps). 
But in the present paper we exploit
the fact that we have maps with different topological properties (as in, for example, \cite{AK}).  

The (inductive) analysis consists of two main parts: an arithmetic one, and a geometric/topologic one. 
The arithmetic part is the more standard one, and is very similar to the one we used in \cite{B1}. The structure of several of the estimates that arise 
in this part are rather general and naturally arise in the analysis of non-uniformly hyperbolic maps.    
However, in the present case, compared with, for example, \cite{B1,B2}, the topology is very different. This allows us to handle the "resonances" 
without using parameter exclusion (as is needed in \cite{B2,Ja, Y2}) and without "losing derivatives" (which one does, for example, in \cite{B1}). 
In the non-resonant case, similar techniques were developed by Young in \cite{Y2}. 
All these techniques, in turn, are based on
ideas from one-dimensional dynamics, in particular \cite{BC, CE}. 

In the final section, Section 8, we have collected "abstract" results, most of them of a computational 
nature, which are used in different parts of our construction.

\end{section}

\begin{section}{Preliminaries}
In this section we introduce some of the notations which we will be using. We also derive elementary properties of
the map $T$ (defined in (\ref{MapT})), assuming that $g$ and $f$ satisfy (A1) and (A2), respectively. 
In the last subsection we will derive some elementary topological facts which will be frequently used to control the geometry. 

From now on we assume that $\omega$ is fixed, satisfying (\ref{DC}) for some $\gamma>0$. Furthermore, we assume that 
$\kappa,\rho>1$ are fixed, and we fix $g$ satisfying condition $(A1)(\kappa)$. We also assume that 
$f$ satisfy condition $(A2)(\ve,\rho)$, where $0<\ve\ll 1$, and the smallness of $\ve$ is only allowed to depend on $\kappa, \rho$ and $\gamma$.

\subsection{Some notations}

\begin{itemize}

\item By $d$ we denote the standard metric on $\T$.

\item If $X,Y\subset \T$ are sets, we define the distance $\dist(X,Y)=\inf_{x\in X, y\in Y}d(x,y)$.

\item If $I\subset \T$, we denote the interior of $I$ by $\inn(I)$; by $|I|$ we denote the (Haar) measure of the set $I$ (provided that
$I$ is measurable).

\item If $I=[x-\delta,x+\delta]\subset \T$ is an interval, we denote by $3I$ the interval $3I=[x-3\delta,x+3\delta]$. Thus,
the two intervals have the same center, and $|3I|=3|I|$ (if $I$ is sufficiently small). 

\item We let $\pi:\mathbb{R}\to \T$ be the natural projection $\pi(x)=x \mod 1$.

\item If $y,\eta\in \T$, the  (positively oriented) interval (or arc) $[y,\eta]$ is defined in the natural way: 
Let $\widetilde{y},\widetilde{\eta}\in [0,1)$ be the unique points
such that $\pi(\widetilde{y})=y$ and $\pi(\widetilde{\eta})=\eta$. If $\widetilde{y}\leq \widetilde{\eta}$, then $[y,\eta]=\pi([\widetilde{y},\widetilde{\eta}])$;
if $\widetilde{y} > \widetilde{\eta}$, then $[y,\eta]=\pi([\widetilde{y},1+\widetilde{\eta}])$.

\item If $(x,y)\in\T^2$, then we define the projections $\pi_1,\pi_2$ by
$$
\pi_1(x,y)=x,\quad \pi_2(x,y)=y.
$$


\item We say that a function $f:I\to \T$ ($I\subset \T$  an interval) is (strictly) increasing if the lifts of $f$ are (strictly) increasing.  

\item If $I\subset \T$ is any set and $x\in \T$, we let
$N(x;I)$ denote the smallest integer $k \geq 0$ such that $x+k\omega\in I$. If no such integer exists, we let
$N(x;I)=\infty$. Note that $N(x;I)=0$ if $x\in I$. Thus, $N$ is the first entry time of $x$ to $I$.

\item If $r>0$ is a real number, we denote by $[r]$ the integer part of $r$. Thus $r-1<[r]\leq r$.


\end{itemize}

\subsection{Definition of the sets $A', A'', R$ and the point $\beta$} 
We now define some sets which play a central r\^ole in the analysis.
Let  $$R=f(\T\setminus A).$$ Since $A$ is open, we have that $R$ is a closed interval. Moreover, 
by assumption (A2) we have $|R|<\ve$.

By assumption (A2) we have $A\subset \T\setminus B$, and thus $|A|<\ve$.
To have some space around the interval $A$, we let $A'$ be the closed set 
\begin{equation}\label{ap}
A'=\{y\in\T: \dist(y,A)\leq \ve\}.
\end{equation}
Since $|A|<\ve$ we have $|A'|< 3\ve$. Moreover, since we clearly have $|A'|>2\ve$, and since $A\subset A'$, 
it follows that
\begin{equation}\label{P_eq1}
\overline{\T\setminus A'}\subset B.
\end{equation}
At a few occasions we will also need an intermediate set, $A''$, which we define by
\begin{equation}\label{app}
A''=\{y\in\T: \dist(y,A)\leq \ve/2\}.
\end{equation}
By definition we thus have $A\subset A''\subset A'$.

We will also use a "reference point" $\beta\in \T\setminus A$ close to $A$. Writing $A=(a_1,a_2)$, we let
\begin{equation}\label{beta}
\beta=a_2+\ve/4\in A''.
\end{equation}
From this choice we immediately get the following (see Figure \ref{fig1}, and recall the definition of the oriented intervals $[y,\eta]$ above):
\begin{lemma}\label{P_betalemma}
If $y\in \T\setminus A''$ then $[\beta,y]\subset \T\setminus A$.
\end{lemma}

\begin{figure}
\psfrag{a}{$A$}
\psfrag{b}{$A'$}
\psfrag{c}{$A''$}
\psfrag{d}{$\beta$}
\psfrag{e}{$y$}
\includegraphics[width=6cm]{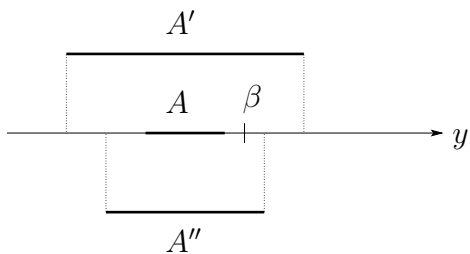}
\caption{The sets $A,A'$ and $A''$, and the point $\beta$}\label{fig1}
\end{figure}

\subsection{Definition of the strip $S$ and the set $I_0$} We recall that the map $T$ is given by
$$T(x,y)=(x+\omega,g(x)+f(y)).$$ By using the definition of the (small) interval $R$ above, we see that if
$(x,y)\in\T\times (\T\setminus A)$, then $T(x,y)\in \{x+\omega\}\times (R+g(x))$. Thus, as long as the iterates 
$T^k(x,y)\in \T\times (\T\setminus A)$ (for $k\geq 0$), we know that they must be in the strip $S$ defined as
\begin{equation}\label{S}
S=\bigcup_{x\in\T}\{x\}\times (R+g(x-\omega))=\bigcup_{x\in\T}\{x+\omega\}\times (R+g(x)).
\end{equation}
\begin{figure}
\psfrag{a}{$\T\times A'$}
\psfrag{b}{$S$}
\psfrag{c}{$I_0^1+\omega$}
\psfrag{d}{$I_0^2+\omega$}
\psfrag{e}{$x$}
\psfrag{f}{$y$}
\includegraphics[width=6cm]{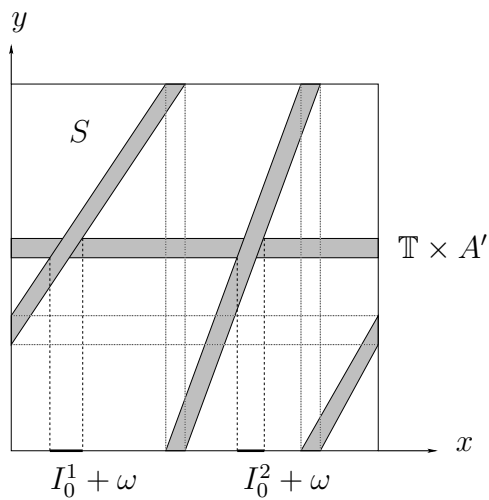}
\caption{The strips $S$ and $\T\times A'$, and the two intervals $I_0^1+\omega$ and $I_0^2+\omega$.}\label{fig2}
\end{figure}
If the iterates stayed in this strip for ever (which they cannot), we would easily have a good control on the dynamics. However, sooner
or later the iterates will enter the strip $\T\times A$; and on $A$ the map $f$ is strongly expanding (and $|f(A)|$ is close to 1 ), and thus uncertainty
is added. Indeed, since $g(\T)=\T$ we see that there must be a connection between the strip $S$ and the strip $\T\times A$ (see Figure \ref{fig2}). 
From the assumption that $g$ has degree 2, there are in fact two such connections (see also Lemma \ref{P_LC} below). 
This property, that there are two connections, and not only one, can cause resonance problems 
(on infinitely small scales) which complicates the analysis.

Since the location of the intersection between the two strips are central for the analysis, we 
define the set $I_0\subset \T$ so that $I_0+\omega$ is the projection onto the 
$x$-coordinate of the intersection between $S$ and the strip $\T\times A'$ (we use the "thickened" interval $A'$ instead of $A$ to have some more space):
\begin{equation}\label{I0}
I_0=\{x\in\T: (R+g(x))\cap A'\neq \emptyset\}=\pi_1(S\cap(\T\times A'))-\omega.
\end{equation}

From the above definitions, together with the assumptions on $f$ and $g$, we get the following results:
\begin{lemma}\label{P_LA}
(1) If $(x,y)\in S$ and $x\notin I_0+\omega$, then $y\in \T\setminus A'$.
(2) If $(x,y)\in S$ and $x\notin \inn(I_0+\omega)$, then $y\in  \overline{\T\setminus A'}\subset B$.
\end{lemma}
\begin{proof} This follows directly from the definition of $S$ and (\ref{P_eq1}).
\end{proof}

Moreover, the set $I_0$ is small (see Figure \ref{fig2}):
\begin{lemma}\label{P_LC}
The set $I_0$ consists of two disjoint intervals, $I_0^1, I_0^2$, each of a length $\sim \ve$.
\end{lemma}
\begin{proof} For completeness we include some details. By assumption we have $|R|<\ve$, and we have $2\ve<|A'|<3\ve$.
Let $G(x)$ be a lift of $g(x)$. Since $g$ has degree 2 we have $G(x+1)=G(x)+2$. From the fact that the equation $g(x)=y$ 
has exactly two solutions $x_1,x_2\in \T$ for each $y\in \T$, and since $A'$ is a single interval and $|A'|\ll 1$, it is clear that $I_0$ consist of two intervals.
Moreover, from the assumptions on $g$ (condition $(A1)(\kappa)$) we have $\frac{1}{\kappa}(\xi-x)<G(\xi)-G(x)<\kappa(\xi-x)$ for all $x<\xi$. Thus, if 
$G(\xi)-G(x)=|A'|+|R|$ we have $(|A'|+|R|)/\kappa<\xi-x<\kappa(|A'|+|R|)$. This gives the bound on the length of each of the two intervals in $I_0$.
\end{proof}

The last result is an elementary fact which will be used at certain places later in the paper.
\begin{lemma}\label{P_LB}
If $I\subset \T$ is an interval such that $(I_0^1+\omega)\subset I$, and $I\cap (I_0^2+\omega)=\emptyset$, 
and if $\vf:I\to\T$ is continuous and strictly increasing, and such that $(x,\vf(x))\in S$ for all $x\in I$, 
then $\vf^{-1}(A')$ is a single interval, and $\vf(\vf^{-1}(A'))=A'$.
\end{lemma}
\begin{proof}Since $(I_0^1+\omega)\subset I$, and $(x,\vf(x))\in S$ for all $x\in I$, it follows from the definition of $I_0$ 
that $\vf(I_0^1+\omega)\supset A'$. Since $\vf$ is strictly increasing it thus follows that $\vf^{-1}(A')\cap (I_0^1+\omega)$ is a single interval.
By the definition of $I_0$ we must have $\vf^{-1}(A')\subset I_0+\omega$. By assumption we have $I\cap (I_0^2+\omega)=\emptyset$; thus
$\vf^{-1}(A')\subset I_0^1+\omega$. See Figure \ref{fig2}.
\end{proof}

\subsection{Basic properties of the map $T$}\label{BPT}

Here we collect some basic information about iterations of the map $T$. We recall the notation
$$
(x_k,y_k)=T^k(x,y), k\in\Z.
$$

We begin with an obvious fact. However, since it is central for the analysis, we state it as a lemma.

\begin{lemma}\label{P_L0}
If $y,\eta\in \T$ and $x\in \T$, then $T(x,[y,\eta])=\{x+\omega\}\times [y_1,\eta_1]$.
\end{lemma}
\begin{proof}Since $f$ preserves the orientation we have $f([y,\eta])=[y_1,\eta_1]$; here 
$y_1=\pi_2(T(x,y))$ and $\eta_1=\pi_2(T(x,\eta))$.
\end{proof}

The statement of the next lemma was included in the discussion prior to the definition of the strip $S$; but for later
references we state it here as a lemma.
\begin{lemma}\label{P_L1}
If $(x,y)\in \T\times (\T\setminus A)$ then $T(x,y)\in S$.
\end{lemma}
\begin{proof}Since $y\notin A$ it follows by definition that $f(y)\in R$; hence $g(x)+f(y)\in R+g(x)$. Thus
$
T(x,y)=(x+\omega, g(x)+f(y))\in S.
$
\end{proof}
The previous lemma, combined with Lemma \ref{P_LA}, gives us
\begin{lemma}\label{P_L2}
If $(x,y)\in S$ and $x\notin \inn(I_0+\omega)$ then $T(x,y)\in S$. 
\end{lemma}







The last lemma in this subsection deals with backward iteration of $T$. 
\begin{lemma}\label{P_L4}
If $(x,y)$ are such that $x_1\notin I_0+\omega$ and $y_1\in A'$, then $y\in A$.
\end{lemma}
\begin{proof}Assume that $y\in \T\setminus A$. Then $(x_1,y_1)\in S$ by Lemma \ref{P_L1}. Since $y_1\in A'$, it follows from 
the definition of $I_0$ that $x_1\in I_0+\omega$. Contradiction.
\end{proof}



\subsection{Some topological facts}\label{stf}

Here we collect some elementary topological facts which we will use to control the geometry in our analysis. The lemmas below
are tailored for the exact situation which we will have later.

First we introduce the following notation (see Figure \ref{fig3}).
\begin{defi}
Assume that $I=[s,t]\subset \T$ is an interval and that the functions $\vf,\psi:I\to\mathbb{T}$ are continuous.
Assume further that $\psi(\xi)=\vp(\xi)$ for some $\xi\in I$, and that the lifts $\tvp,\tpsi$ are chosen so that $\tvp(\xi)=\tpsi(\xi)$.
If $\tpsi(x)<\tvp(\xi)$ for all $x<\xi$ sufficiently close to $\xi$, and if $\tpsi(x)>\tvp(\xi)$ for all $x>\xi$ sufficiently close to $\xi$,
we say that $\psi$ \emph{overtakes $\vp$ at $\xi$}. 
\end{defi}

\begin{figure}
\psfrag{b}{$y=\psi(x)$}
\psfrag{a}{$y=\vp(x)$}
\psfrag{c}{$I$}
\psfrag{d}{$\xi$}
\includegraphics[width=4cm]{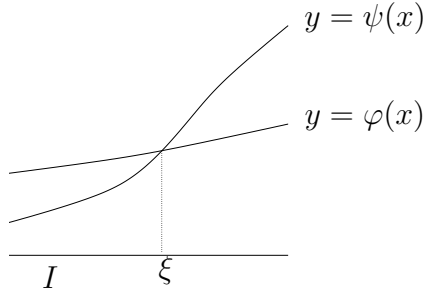}
\caption{$\psi$ overtakes $\vp$ at $\xi$.}\label{fig3}
\end{figure}

An obvious case when this happens is when $\psi$ is strictly increasing, and $\vf$ is a constant function. Then $\psi$ overtakes
$\vf$ at each point where the graphs intersect.

Since the map $f:\T\to\T$ preserves orientation, we immediately get:
\begin{lemma}\label{Tf_0}
Assume that $I=[s,t]\subset \T$ is an interval and that the functions $\vf,\psi:I\to\mathbb{T}$ are continuous.
Assume further that $\psi$ overtakes $\vf$ at $x=\xi_i$ ($i=1,2,\ldots,k$), and that $\psi(x)\neq\vf(x)$ for all $x\neq \xi_i$.
Then the same holds for $\psi_1$ and $\vf_1$, where $\psi_1(x)=\pi_2(T(x,\psi(x)))$ and $\vf_1(x)=\pi_2(T(x,\vf(x)))$.
\end{lemma}
\begin{proof}Note that $\psi_1(x)=g(x)+f(\psi(x))$ and $\vf_1(x)=g(x)+f(\vf(x))$, and recall that $f$ is strictly increasing.
\end{proof}




We recall the definition of the point $\beta$ in (\ref{beta}) and that of $A''\subset A'$ in (\ref{app}).
\begin{lemma}\label{Tf_1}
Assume that $I=[s,t]\subset \T$ is an interval and that the function $\psi:I\to\mathbb{T}$ is continuous and strictly increasing.
If $\psi^{-1}(A'')$ consists of one interval , $I^1$, and $\psi(I^1)=A''$, 
then $[\beta,\psi(s)]\subset \T\setminus A$ and $[\beta,\psi(t)]\subset \T\setminus A$. Moreover, there is a unique $\xi\in I^1$
such that $\psi(\xi)=\beta$, this $\xi\in \inn(I^1)$ (and, trivially, $\psi$ overtakes the constant function $\vf(x)=\beta$ at $\xi$).  
\end{lemma}
\begin{proof}Easy. See Figure \ref{fig4}
\end{proof}

\begin{figure}
\centering
 \begin{minipage}{0.4\textwidth}
\psfrag{b}{$\beta$}
\psfrag{i}{$\psi^{-1}(A'')$}
\psfrag{A}{$A$}
\psfrag{Ap}{$A''$}
\psfrag{int1}{$[\beta,\psi(t)]$}
\psfrag{int2}{$[\beta,\psi(s)]$}
\psfrag{y}{$y=\psi(x)$}
\psfrag{s}{$s$}
\psfrag{t}{$t$}
\psfrag{xi}{$\xi$}
\includegraphics[width=6.5cm]{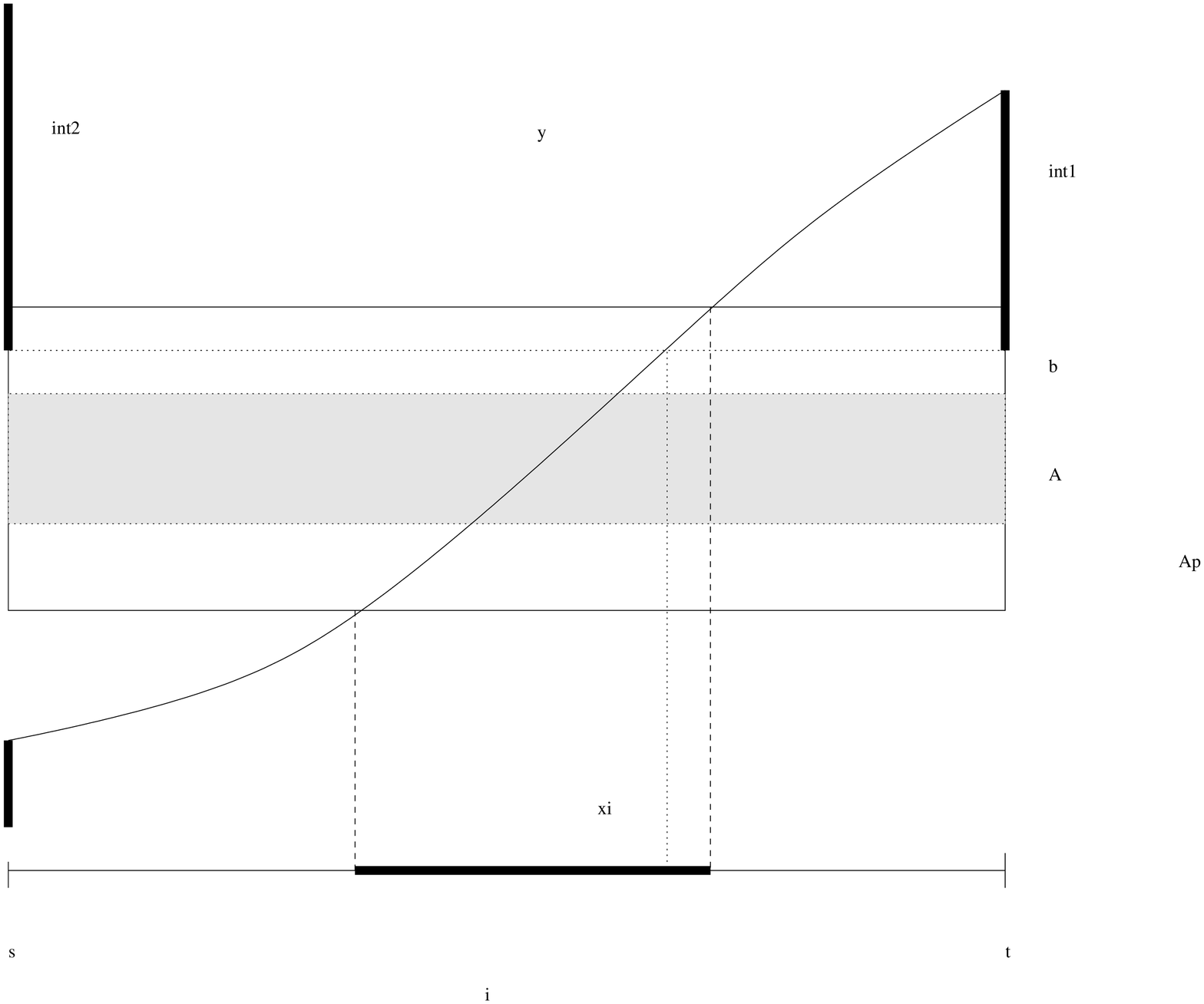}
\caption{The statement of Lemma \ref{Tf_1}.}\label{fig4}
\end{minipage}
\hfill
 \begin{minipage}{0.4\textwidth}
\psfrag{b}{$\beta$}
\psfrag{I}{$\vf^{-1}(A')$}
\psfrag{A}{$A$}
\psfrag{Ap}{$A'$}
\psfrag{y1}{$y=\psi(x)$}
\psfrag{y2}{$y=\vf(x)$}
\psfrag{x}{$\xi$}
\includegraphics[width=5cm]{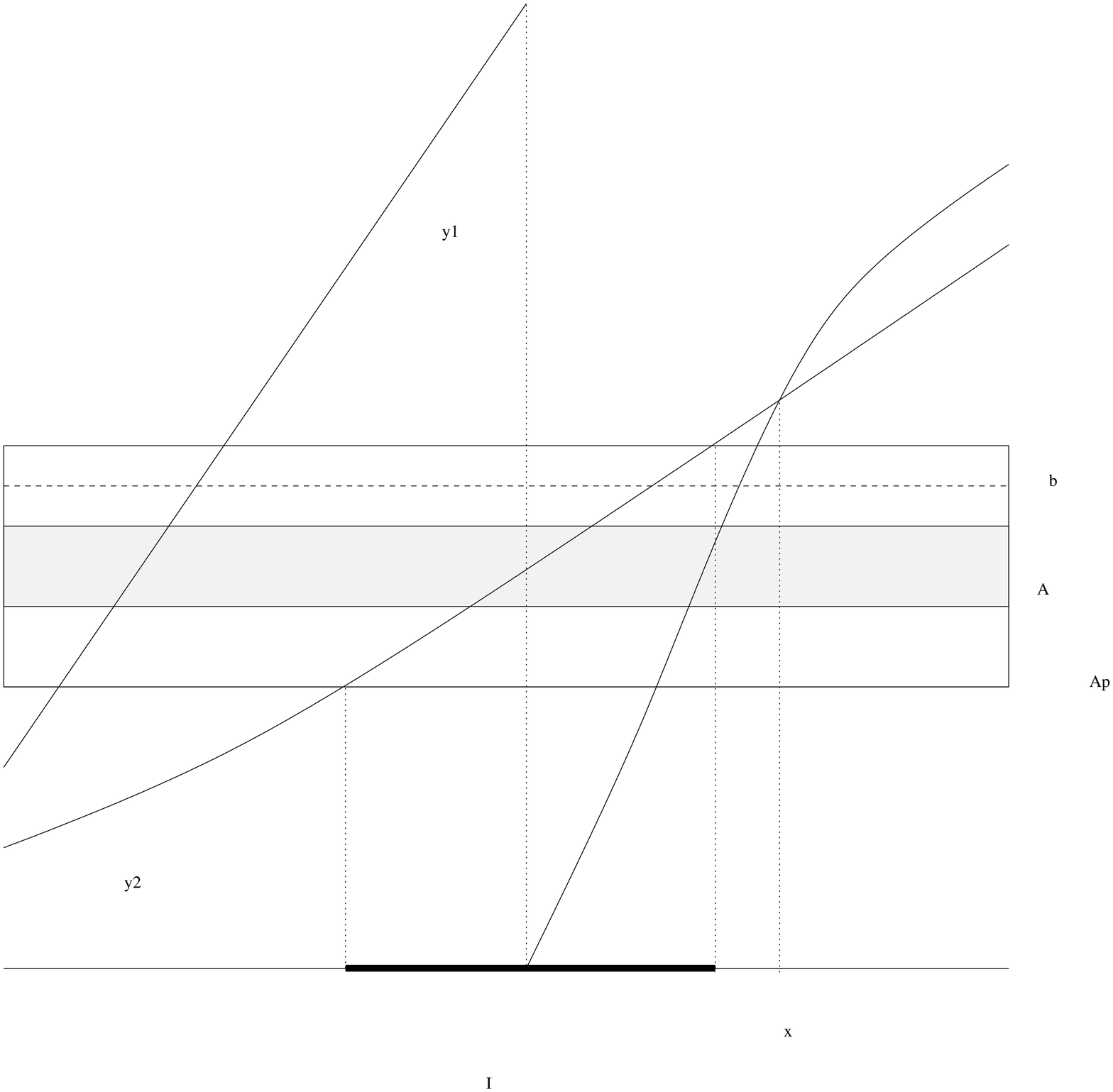}
\caption{Figure for Lemma \ref{Tf_3}.}\label{fig5}
\end{minipage}
\end{figure}

\begin{lemma}\label{Tf_3}
Assume that $I=[s,t]\subset \T$ is an interval and that the functions $\vf,\psi:I\to\mathbb{T}$ are continuous and strictly increasing,
and that $\vf^{-1}(A')$ consists of one interval, $I^0$, and $\vf(I^0)=A'$. Assume further that
$[\vf(s),\psi(s)], [\vf(t),\psi(t)]\subset \T\setminus A$, and that there is a unique $\xi\in I$ such that
$\vf(\xi)=\psi(\xi)$, $\xi\in\inn(I)$, and that $\psi$ overtakes $\vf$ at $\xi$. 
Then the equation $\psi(x)=\beta$ has exactly two solutions $\xi_1,\xi_2\in I$, these $\xi_1,\xi_2\in\inn(I)$ (and, trivially, $\psi$ overtakes the constant function
$\phi(x)=\beta$ at $\xi_1,\xi_2$). Furthermore, $[\beta,\psi(s)]\subset \T\setminus A$ and $[\beta,\psi(t)]\subset \T\setminus A$.
\end{lemma}

\begin{proof} We assume that $I=[s,t]\subset \R$. 
Let $[a_1',a_2']\subset \R$ be such that $\pi([a'_1,a'_2])=A'$,  and let $a_1'<a_1<a_2<a_2'$ be such that
$\pi((a_1,a_2))=A$. We also let $a_2<b<a_2'$ be such that $\pi(b)=\beta$ (recall the definition of $\beta$ in (\ref{beta})).

Let $\tvp$ be the lift of $\vf$
such that $a_2'<\tvp(s)\leq a'_1+1$. From the assumptions we must have $a'_2+1\leq \tvp(t)<a'_1+2$. We let $\tpsi$ be the lift of $\psi$
such that $\tvp(s)<\tpsi(s)\leq a_1+1$ (recall that $[\vf(s),\psi(s)]\subset \T\setminus A$). Note that $a_2<b<\tpsi(s)\leq a_1+1$, and thus 
$[\beta,\psi(s)]\subset \T\setminus A$.

Since there is only one $\xi\in I$ such that $\vf(\xi)=\psi(\xi)$, and since this
$\xi$ is in the interior of $I$, and $\psi$ overtakes $\vf$ at $\xi$, we must have $\tpsi(\xi)=\tvp(\xi)+1$, and
$\tvp(x)+1<\tpsi(x)<\tvp(x)+2$ for all $x>\xi$. The assumption
that $[\vf(t),\psi(t)]\subset \T\setminus A$ therefore gives us that $\tvp(t)+1<\tpsi(t)<a_1+3$. 
Since $b+2<a_2'+2<\tvp(t)+1$, this last inequality shows that $[\beta,\psi(t)]\subset \T\setminus A$. Moreover, we note that 
$[a_1+1,a_2'+2]\subset \tpsi(I)\subset (a_2',a_1+3)$. Since $b+1,b+2\in [a_1+1,a_2'+2]$, and $b,b+3\notin (a_2',a_1+3)$, and since 
$\tpsi$ is strictly increasing, this shows that the equation $\psi(x)=\beta$ has exactly two solutions $\xi_1,\xi_2\in I$, and that $\xi_1,\xi_2\in\inn(I)$.
\end{proof}

\begin{lemma}\label{Tf_2}
Assume that $I=[s,t]\subset \T$ is an interval and that the 
functions $\vf,\psi:I\to\mathbb{T}$ are continuous and strictly increasing. 
Assume further that $\vf(x)\in \T\setminus A'$ for all $x\in I$, and  
$[\vf(s),\psi(s)], [\vf(t),\psi(t)]\subset \T\setminus A'$.

\begin{enumerate}
\item If there is a unique $\xi\in I$ such that $\psi(\xi)=\vf(\xi)$, if $\xi\in \inn(I)$, and if $\psi$ overtakes $\vf$ at $\xi$, 
then $\psi^{-1}(A')$ consists of one single interval, $I'$, and $\psi(I')=A'$.

\item If there are exactly two solutions $\xi_1,\xi_2$ to $\psi(x)=\vf(x)$,  if $\xi_1,\xi_2\in \inn(I)$, and if $\psi$ overtakes $\vf$ at $\xi_i$ ($i=1,2$), 
then  
$\psi^{-1}(A')$ consists of two interval, $I', I''$, and $\psi(I')=\psi(I'')=A'$.
\end{enumerate}
 
\end{lemma}

\begin{figure}
\psfrag{i}{$\psi^{-1}(A')$}
\psfrag{A}{$A'$}
\psfrag{y1}{$y=\psi(x)$}
\psfrag{y2}{$y=\vf(x)$}
\psfrag{x}{$\xi$}
\includegraphics[width=5cm]{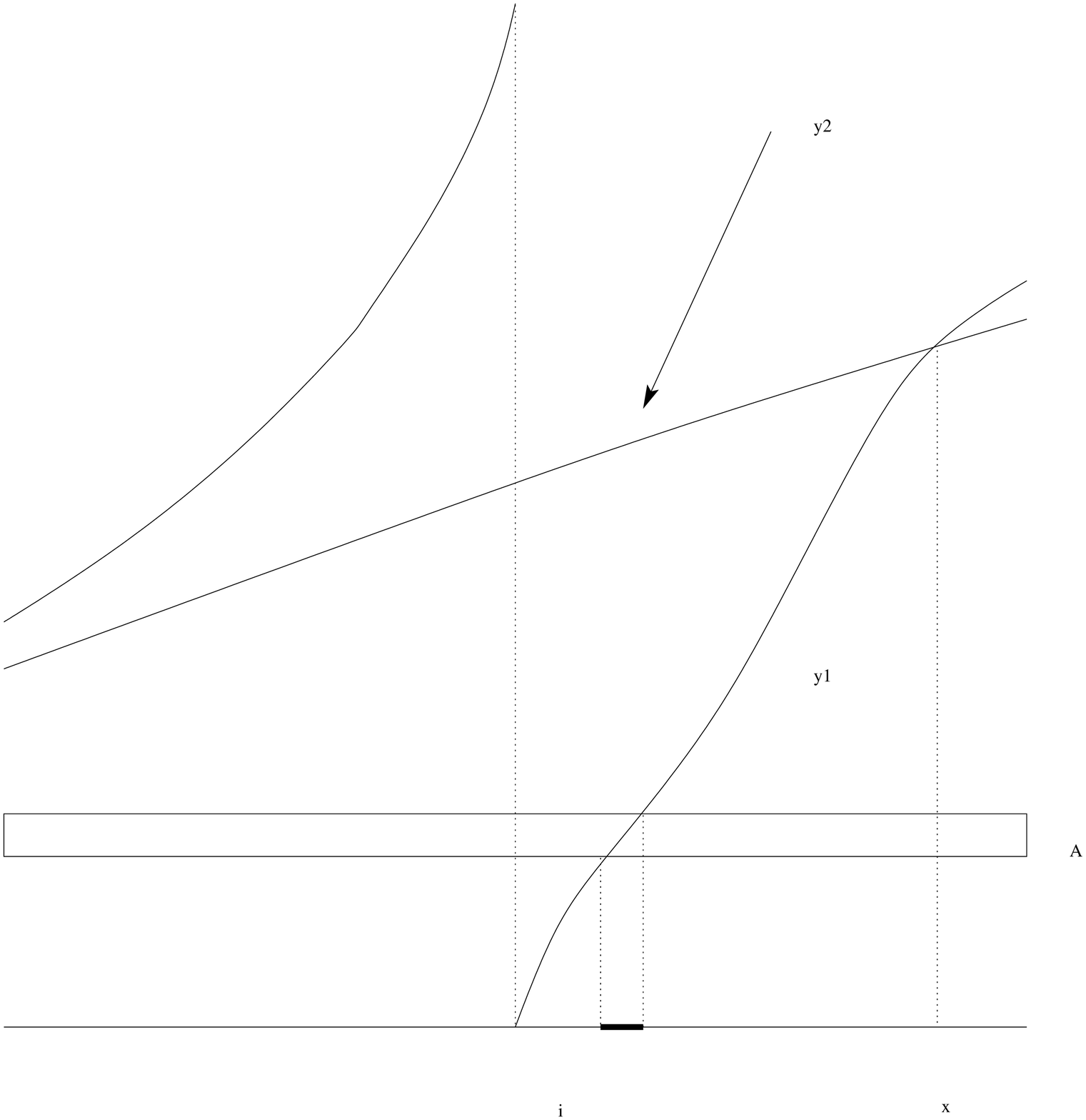}
\caption{Figure for Lemma \ref{Tf_2}.}\label{fig6}
\end{figure}

\begin{proof} We prove (1); the proof of (2) is similar.  Let $[a_1,a_2]\subset \R$ be such that $\pi([a_1,a_2])=A'$. 
By assumption we have $\vp(I)\cap A'=\emptyset$. Assume that $I=[s,t]\subset \R$. Let $\tvp$ be the lift of $\vf$
such that $a_2<\tvp(x)<a_1+1$ for all $x$, and let $\tpsi$ be the lift of $\psi$ such that $\tvp(s)<\tpsi(s)<a_1+1$ (this is possible since 
$[\vf(s),\psi(s)]\cap A'=\emptyset$). Since there is a unique $\xi\in I$ such that $\psi(\xi)=\vf(\xi)$, this point being in the interior of $I$, and
since $\psi$ overtakes $\vf$ at $\xi$, we must have $\tvp(x)<\tpsi(x)$ for all $x\in I$. From this we conclude that we must have 
$\tpsi(\xi)=\tvp(\xi)+1$, and $\tpsi(x)>\tvp(x)+1$ for all $x>\xi$. Since $\tvp(t)+1<a_1+2$,  since $[\vf(t),\psi(t)]\cap A'=\emptyset$,
and since $\tpsi(t)<\tvp(t)+2$ (otherwise the equation $\psi(x)=\vf(x)$ would have more than one solution), we see that we must have $\tpsi(t)<a_1+2$.
From this the statement in (1) follows.

\end{proof}

\end{section}

\begin{section}{Some formulae}\label{formulae}
To control the dynamics of the map $T$ we need to know how the iterates $(x_k,y_k)=T^k(x,y)$ depend on the initial point $(x,y)$. In this section
we derive various formulae for this, as well as quantitative estimates. 

We first collect a few well-known facts about Lipschitz functions.
\subsection{Some properties of Lipschitz continuous functions}\label{lipschitz}

Assume that $\vf,\psi:\T\to\T$ are Lipschitz continuous. Since Lipschitz continuity implies absolutely continuity, it follows that
$\vf$ and $\psi$ are differentiable a.e.; and if $\tvp$ is a lift of $\vf$, then the formula
$$
\tvp(b)-\tvp(a)=\int_a^b \tvp'(x)dx
$$ 
holds. 

The sum $\vf+\psi$ is clearly Lipschitz continuous. Moreover, the composition $\vf\circ \psi$ is again Lipschitz continuous, and the chain rule holds, i.e., 
$$(\vf\circ \psi)'(x)=\vf'(\psi(x))\psi'(x) \text{ for a.e. } x\in \T, $$
where $\vf'(\psi(x))\psi'(x)$ is interpreted to be zero whenever $\psi'(x)=0$. 

If $\vf$ has a Lipschitz constant $C$, then $|\vf'(x)|\leq C$  whenever the derivative exists; and if $\vf$ is bi-Lipschitz, where $\vf$ and $\vf^{-1}$ both
have Lipschitz constant $C>1$, then $C^{-1}\leq|\vf'(x)|\leq C$.

We will sometimes use the following notation: if $\varphi:I\to \T$ is Lipschitz continuous, 
then "$|\varphi'(x)|\leq C$ on $I$" means that $|\varphi'(x)|\leq C$ whenever the derivative exists, i.e., for a.e. $x\in I$. 
Note that this implies, by the integral formula above, that $\varphi$ has a Lipschitz constant $C$.

From the assumptions on $f$ and $g$  in (A1-2) we have the global bounds
$$
\frac{1}{\kappa}<g'(x)<\kappa \text{ for a.e. } x\in \T, \quad \ve^\rho<f'(y)<\ve^{-\rho} \text{ for a.e. } y\in \T;
$$
and on the intervals $A\subset \T$ and $B\subset \T$ we have
$$
f'(y)<\ve \text{ on } B; \quad
f'(y)>\ve^{-1}  \text{ on } A.
$$



\subsection{Formulae} For a fixed $x\in \T$ we consider the iterates $y_k:\T\to\T$ defined by
$$
y_k(y)=\pi_2(T^k(x,y)).
$$
Each of the iterates $y_1(y)=g(x)+f(y)$, $y_2(y)=g(x+\omega)+f(y_1(y)), \ldots$ are orientation preserving circle homeomorphisms, and they are all
Lipschitz continuous by the properties in the previous subsections. We have $y_1'(y)=f'(y)$, $y_2'(y)=f'(y_1)f'(y)$, and generally
$$
y_{k+1}'(y)=\prod_{j=0}^k f'(y_j(y)) \text{ for a.e. } y\in\T.
$$
Furthermore, if $Y\subset \T$ is an interval, then 
$$
|y_k(Y)|=\int_{Y}y_k'(y)dy.
$$

Below we shall often keep track on products (for fixed $x$) 
of the form $\prod_{j=0}^k f'(y_j(y))$ for a.e. $y$ in some interval $Y$; and our formulae will contain such expressions.
We note that $y_{k+1}(y)$ having a Lipschitz constant $C$ on $Y$ is equivalent to $\prod_{j=0}^k f'(y_j(y))\leq C$ on $Y$ since we know that $y_{k+1}$ is Lipschitz 
continuous (and $f'(y)$ is always positive).

From these observations we immediately get
\begin{lemma}\label{SF_L1}
Assume that $Y\subset \T$ is an interval, and that for some $k\geq 1$ and $x\in \T$
$$
\prod_{j=0}^{k-1}f'(y_j(y)) \underset{(\geq)}\leq  \delta \quad \text{ on } Y. 
$$                           
Then $|y_k(Y)|\underset{(\geq)}\leq  \delta|Y|$.
\end{lemma}

We shall need formulae to control the iterates (under $T$) of curves $(x,y_0(x))$, $x\in I$, where $I$ is an interval.

\begin{lemma}\label{SF_L2}
Assume that $y_0:I\to \T$  ($I$ an interval) is Lipschitz continuous, and define the functions $y_j(x)$ by
$
y_j(x)=\pi_2(T^j(x,y_0(x))).
$ 
For each $k\geq 0$ the function $y_{k+1}$ is  Lipschitz continuous on $I$ and we have the formula
$$
\begin{aligned}
y_{k+1}'(x)&=g'(x_k)+g'(x_{k-1})f'(y_k)+g'(x_{k-2})\prod_{j=k-1}^k f'(y_j) +\ldots + g'(x)\prod_{j=1}^k f'(y_j) +\\ &+ y_0'(x)\prod_{j=0}^k f'(y_j) 
\text{ for a.e. } x\in I.
\end{aligned}
$$
\end{lemma}
\begin{proof}Note that $y_1(x)=g(x)+f(y_0(x))$; thus $y_1'(x)=g'(x)+f'(y_0)y_0'(x)$.  Next, $y_2(x)=g(x_1)+f(y_1(x))$, and therefore
$y_2'(x)=g'(x_1)+f'(y_1)y_1'(x)=g'(x_1)+f'(y_1)(g'(x)+f'(y_0)y_0'(x))$. The general formula follows by induction.
\end{proof}

As a direct consequence of the previous lemma we get the following result, where we assume that we for each $x\in I$ have various controls on products of the form
$$
\prod_{j=\ell}^k f'(y_j) 
$$
where, as usual, $y_j=\pi_2(T^j(x,y)$.
\begin{lemma}\label{SF_L3}

Below $I\subset \T$ and $Y\subset \T$ are assumed to be intervals. 

\begin{enumerate}

\item Assume that $0<\delta<1$ and $k\geq 1$. Assume further that we for all $x\in I$ have the estimate
$\displaystyle
\prod_{j=\ell}^k f'(y_j)<\delta^{k-l+1} \text{ on } Y \text{ for all } \ell\in [0,k].
$
If $y_0:I\to Y$ is Lipschitz continuous and increasing, then 
$$
y_{k+1}'(x)<\frac{\kappa}{1-\delta}+ \delta^{k+1}y_0'(x) \text{ on } I.
$$

\item Assume that $\Delta>0$. Moreover, assume that for all $x\in I$ we have 
$\displaystyle
\prod_{j=0}^k f'(y_j)>\Delta 
$
for all $y$ such that $y_{k+1}\in Y$ (and for which the derivatives exist). If $y_0:I\to \T$ is Lipschitz continuous and increasing, and if 
$y_{k+1}(I)\subset Y$, then
$$
y_{k+1}'(x)> \Delta y_0'(x) \text{ on } I.
$$

\item For all increasing Lipschitz functions $y_0:\T\to\T$ we always have the (trivial) upper bound
$$
y_{k+1}'(x)<\frac{\kappa\ve^{-\rho k}}{1-\ve^\rho}+\ve^{-\rho(k+1)}y_0'(x) \text{ on } I, k\geq 0.
$$ 
\item For all increasing Lipschitz functions $y_0:\T\to\T$ always have the (trivial) lower bound $y_{k+1}'(x)>1/\kappa$ on $I$, $k\geq 0$.

\end{enumerate}

\end{lemma}

\begin{proof}
(1) From the formula in Lemma \ref{SF_L2}, combined with the assumed estimates and the general assumption that $g'(x)<\kappa$ for a.e. $x$, we get
$$
y_{k+1}'<\kappa(1+\delta+\delta^2+\ldots +\delta^k)+\delta^{k+1}y_0'<\frac{\kappa}{1-\delta}+ \delta^{k+1}y_0'.
$$
(2) Since all the terms in the formula in Lemma \ref{SF_L2} are positive, 
we get, making use of the assumption,
$$
y_{k+1}'>\Delta ~y_0'.
$$
(3) Follows easily from the formula in Lemma \ref{SF_L2} since we have the global upper bounds $f'(y)<\ve^{-\rho}$ and $g'(x)<\kappa$. 

(4) Again, since all the terms in the formula are non-negative, and since $g'(x)>1/\kappa$ for a.e. $x$, the statement is obvious.

\end{proof}

We end this section with an elementary result which we will use later to show that certain sets are "well" inside other sets. We recall the definition of the set $A''$ in
(\ref{app}).
\begin{lemma}\label{SF_L4}
Assume that $I=[s,t]\subset \T$ is an interval and that $\vf:I\to \T$ is Lipschitz continuous, strictly increasing, and such that $\vf'(x)<\Delta$ on $I$.
Assume further that $\vf(s),\vf(t)\notin A''$.
Then $\dist(\vf^{-1}(A),\partial I)>\ve/(2\Delta)$.
\end{lemma}
\begin{proof}Let $\tvp:\R\to\R$ be a lift of $\vf$, and assume that $I=[s,t]\subset \R$. Let $s<s'$ be the smallest $s'$ such that
$\pi(\tvp(s'))\in \overline{A}$ (we assume that $\vf^{-1}(A)\neq \emptyset$). From the assumption that $\vf(s)\notin A''$, it follows that $\tvp(s')-\tvp(s)> \ve/2$. Moreover, 
$\tvp(s')-\tvp(s)=\int_s^{s'}\tvp'(x)dx\leq \Delta(s'-s)$. Thus, $\dist(\varphi^{-1}(A),s)>\ve/(2\Delta)$. In the same way we can prove that
$\dist(\varphi^{-1}(A),t)>\ve/(2\Delta)$.

\end{proof}

\end{section}

\begin{section}{Proof of Theorem 1 (assuming good control on the dynamics of $T$)}
In this section we prove Theorem 1, assuming that we have very good control of the dynamics of the map $T$. The estimates we will use are
contained in the next proposition, Proposition \ref{pt_prop}. The proof of this proposition is the content of sections 5 and 6 below.  

Before stating the proposition we recall the notation: given $(x,y)\in \T^2$ we define $(x_n,y_n)$ by $$(x_n,y_n)=T^n(x,y), ~n\in\Z.$$
We also recall that we have fixed $\omega$ satisfying the Diophantine condition (\ref{DC}) for some $\gamma>0$, and that we have assumed that 
$g$ satisfies condition $(A1)(\kappa)$ and $f$ satisfies $(A2)(\ve,\rho)$.  
Finally, we recall that $A\subset A'$ (see the definition of $A'$ in (\ref{ap})).

\begin{prop}\label{pt_prop}
There exists an $\ve_1=\ve_1(\gamma,\kappa,\rho)>0$ such that for every $0<\ve<\ve_1$ there is a measurable set $X\subset \T$, $|X|>0$, such that
the following holds for each $x\in X$:
\begin{enumerate}
\item For a.e. $y\in \T\setminus A$ we have 
$$
\prod_{j=1}^{k}f'(y_j)<\ve^{k/2} \text{ for all } k\geq 1. 
$$

\item For all $y\in \T\setminus A$ we have that whenever $x_j\in X$ ($j\geq 1$), then $y_j\in \T\setminus A'$. Moreover,
$\overline{\{(x_k,y_k)\}_{k\geq 0}}=\T^2$. 

\item If $Y\subset \T$ is an interval such that $(x_k,y_k)\in X\times A'$ for all $y\in Y$, any $k\geq 1$, then
$$
\prod_{j=0}^{k-1}f'(y_j)>\ve^{-k/2} \text{ for a.e. } y\in Y.
$$
\end{enumerate}
\end{prop}
\begin{rem}We could have replaced condition (3) with a condition similar to that of (1), but instead considering backward iterations.
However, since we do not really need such a condition we have chosen to work with the above form.
\end{rem}

\begin{proof}[Proof of Theorem 1] 
We shall first construct two measurable functions $u:\T\to\T$ and $s:\T\to\T$, whose graphs are a.e. $T$-invariant. 
Moreover, we will see that for a.e. $x\in\T$ we have
$$
d(y_k,u(x_k))\to 0 \text{ (exponentially fast) as } k\to \infty \text{ for all } y\neq s(x).
$$
($u,s$ are the equivalent of the (projective) Oseledets' directions in the linear cocycle case. See \cite{O}.)

Let $$
X^*=\{x\in \T: x+j\omega\in X \text{ for infinitely many } j\geq 0 \text{ and infinitely many } j<0 \}.
$$
Since $|X|>0$ it follows immediately from the Poincar\'e recurrence theorem that $X^*$ has full (Lebesgue) measure; and
obviously $X^*$ is invariant under rotation by $\omega$. 
We now construct $u(x),s(x)$ for each $x\in X^*$.

Fix $x\in X^*$ and let $\xi=x_p$, where $p\geq 0$ is the smallest integer such that $x+p\omega\in X$. We shall now focus on iterates of $(\xi,y)$ for $y\in \T$.
Let $0=j_0<j_1<j_2<\ldots$ be the times when $\xi_{j_i}\in X$.

\noindent \emph{Claim 1}: If $y$ is such that $y_{j_i}\in A'$ ($i\geq 1$), then $y_{j_l}\in A$ for all $0\leq l<i$.  

Indeed, if $y_{j_l}\notin A'$ for some $0\leq l<j$, then, since $\xi_{j_l}\in X$, applying Proposition \ref{pt_prop}(2) to $(\xi_{j_l},y_{j_l})$ would give us that
$y_{j_i}\notin A'$.

Define the closed intervals $S_i=\pi_2(T^{-j_i}(\xi_{j_i},A'))$, $i\geq 1$. 
From Claim 1 it follows that $S_{i+1}\subset S_i\subset A$ for all $i\geq 1$. Moreover, from Proposition \ref{pt_prop}(3) combined with
Lemma \ref{SF_L1} it follows that $|S_i|\to 0$ as $i\to\infty$. Let $s(\xi)$ be the unique point in $\bigcap_{i\geq 1}S_i$. 
We now define $s(x)$ by $s(x)=\pi_2(T^{-p}(\xi,s(\xi))$.
Note that, since $s(\xi)\in S_i$ for each $i\geq 1$ (by construction), we have 
$T^{p+j_i}(x,s(x))=T^{j_i}(\xi,s(\xi))\in X\times A' \text{ for all } i\geq 1$. Using Claim 1 we conclude that we in fact have  
\begin{equation}\label{pp_eq1}
T^{p+j_i}(x,s(x))=T^{j_i}(\xi,s(\xi))\in X\times A \text{ for all } i\geq 1. 
\end{equation}

To construct $u(x)$ we do as follows. Let $\theta=x_{-q}$, where $q\geq 0$ is the smallest integer such that $x-q\omega\in X$. 
Let $0>k_1>k_2>\ldots$ be the (negative) times when $\theta_{k_i}\in X$, and define the closed intervals 
$U_i=\pi_2(T^{-k_i}(\theta_{k_i},\T\setminus A))$, $i=1,2,\ldots$. 
Proposition \ref{pt_prop}(2) implies that $U_{i+1}\subset U_i \subset \T\setminus A'$. Furthermore, Proposition \ref{pt_prop}(1), and Lemma \ref{SF_L1},
give that $|U_i|\to 0$ as $i\to \infty$. Let $u(\theta)$ be the unique point in $\bigcap_{i\geq 1}U_i$, and define $u(x)=\pi_2(T^{q}(\theta,u(\theta))$. 
Since $(\theta,u(\theta))\in X\times (\T\setminus A)$, and since $x=\theta+q\omega$,
it follows from Proposition  \ref{pt_prop}(2) that 
\begin{equation}\label{pp_eq2}
T^{p+j_i}(x,u(x))=T^{j_i+q+p}(\theta,u(\theta))\in X\times (\T\setminus A') \text{ for each } i\geq 1. 
\end{equation}

By combining equations (\ref{pp_eq1}) and  (\ref{pp_eq2}) it follows that $s(x)\neq u(x)$ for all $x\in X^*$. 
Moreover, from the above construction of $u(x)$ and $s(x)$
it is clear that $T(x,u(x))=(x+\omega,u(x+\omega))$ and  $T(x,s(x))=(x+\omega,s(x+\omega))$ for all $x\in X^*$. 

Next we show that 
$$
d(y_n,u(x_n))\to 0 \text{ as } n\to\infty \text{ for all } x\in X^*, y\neq s(x).
$$
To do this we first note that Proposition \ref{pt_prop}(1) and Lemma \ref{SF_L1} imply that 
\begin{equation}\label{pp_eq3}
|\pi_2(T^n(z,\T\setminus A))|\to 0 \text{ (exponentially fast) as } n\to\infty  \text{ for all } z\in X.
\end{equation}
Fix $x\in X^*$, and let $p,\xi,j_i,S_i$ be defined as above. Take $y\neq s(x)$. If $y_{p+j_i}\in \T\setminus A'$ for some $i$, then it follows from
(\ref{pp_eq2}) and (\ref{pp_eq3}) that $d(y_n,u(x_n))\to 0$  as  $n\to\infty$. Assume now that $y_{p+j_i}\in A'$ for all $i\geq 1$.
Then, by definition, $y_p\in S_i$ for all $i\geq 1$. But this means that $y_p\in\bigcap_{i\geq1} S_i$, i.e., $y_p=s(\xi)$; hence $y=s(x)$, 
contradicting the assumption.

\medskip
It remains to show that $T$ is minimal. By Proposition $\ref{pt_prop}(2)$ we have
\begin{equation}\label{pp_eq0}
\overline{(x_k,y_k)_{k\geq 0}}=\T^2 \text{ for all } x\in X, y\in \T\setminus A.
\end{equation}
The plan is to use Proposition \ref{M_P} in Section 7 with $w^+=u$ and $w^-=s$ 
(note that all the conditions needed for applying the proposition are fulfilled).  
From that proposition we get that $T$ has exactly two invariant and ergodic Borel probability measures, $\mu^s,\mu^u$. Moreover, $\mu^u$ and $\mu^s$ are
the push-forward of the Lebesgue measure on $\T$ by the maps $x\mapsto (x,u(x))$ and $x\mapsto (x,s(x))$, respectively. Proposition \ref{M_P} also tells us that
$\text{supp }\mu^u$ is a minimal set. From (\ref{pp_eq0}) and (\ref{pp_eq2}) we conclude that $\overline{\bigcup_{k=0}^\infty\{(x+k\omega,u(x+k\omega))\}}=\T^2$ for
all $x\in X^*$. Thus $\text{supp }\mu^u=\T^2$, i.e., $\T^2$ is a minimal set.
\end{proof}

The remaining part of the paper is now devoted to the proof of Proposition \ref{M_P}. The set $X\subset \T$ in 
that proposition will be constructed inductively. We will obtain sets
$$
X_0\supset X_1\supset \ldots \supset X_n\supset \ldots \supset X
$$ 
where we get control on the iteration of points $(x,y)$, $x\in X_n$, for times $T_n\to\infty$. 
The Diophantine condition $(\ref{DC})$ posed on $\omega$ is crucial for the analysis; it is used to obtain stopping times.

\end{section}

\begin{section}{Base case of the construction}
We are now ready to take the fist steps towards the proof of Proposition \ref{pt_prop}. The proof of this proposition is 
the heart of our construction, which is based on an inductive argument. 

In this section we state and prove the base case for the inductive construction. Several of the geometric ideas used in the general 
inductive step are visible already here. The main approach is similar to the one we used in \cite{B1}, but the geometry is very different (due
to the fact that the map $T$ is not homotopic to the identity). This is why, in particular, we do not need much regularity assumptions on $f$ and $g$.

\subsection{Control on iterations on first scale}
From the basic estimates in Section \ref{BPT}, we immediately get good control on the iterates of points $(x,y)\in \T\times (\T\setminus A)$ up to time
$N(x;I_0)$ (recall that this is the first entry time of the point $x$ to $I_0$ under translation by $\omega$). 
We summarize them in the next lemma.

\begin{lemma}\label{FS_L1} 
The following holds:
\begin{enumerate}
\item Assume that $x\in \T$, and let $N=N(x;I_0)$. For all $y\in \T\setminus A$ we have
$y_j\in \T\setminus A'$ for all $j\in [1,N]$ 
and $ (x_j,y_j)\in S$ for all $j\in [1,N+1]$. In particular, for all $1\leq \ell\leq k\leq N$ we have
$$
\prod_{j=\ell}^{k}f'(y_j)<\ve^{(k-l+1)} \text{ on } \T\setminus A.
$$


\item Given $k\geq1$, assume that $x$ is such that $x\in \left(\T\setminus \bigcup_{j=0}^{k-1} (I_0-j\omega)\right)$ and $y$ is such that $y_k\in A'$. 
Then 
$$
y_j\in A, ~j=0,\ldots, k-1, 
$$
and thus 
$$
\prod_{j=0}^{k-1}f'(y_j)>\ve^{-k} 
$$
whenever the derivatives exist.
\end{enumerate} 
\end{lemma}
\begin{proof}(1) By Lemma \ref{P_L1} we have $(x_1,y_1)\in S$. The results in the statement of (1) now follows by
a repeated use of Lemmas \ref{P_LA} and \ref{P_L2}, and by recalling that $f'(\eta)<\ve$ for a.e. $\eta\in B$.

(2) From the assumption on $x$ it follows that $x_j=x+j\omega\notin I_0+\omega$ for all $j=1,2,\ldots, k$. Since $A\subset A'$, and since 
$f'(\eta)>1/\ve$ for a.e. $\eta\in A$, the result follows by applying Lemma \ref{P_L4} repeatedly.
\end{proof}

To obtain control of orbits for longer time than $N(x;I_0)$ we have to analyse what can happen after we have enter $I_0$, and what is the probability for
the different scenarios. 
To do this we will send in "probes" through $I_0$ which we then can follow when we iterate points.  However, 
the first step is to check if the two intervals in $I_0$ are what we call "resonant" or not.

\subsection{Definition of the sets $J_0, \widehat{J}_0$ and the integers $M_0,K_0$.}\label{DJMK}
We know (Lemma \ref{P_LC}) that the set $I_0$ consists of two intervals, $I_0^1$ and $I_0^2$, each of length $\sim \ve$. Now it is time to use the 
Diophantine condition on $\omega$. By
applying Lemma \ref{M_L1} we see that 
\begin{equation}\label{Ch_eq00}
I_0^i\cap(I_0^i+k\omega)=\emptyset \quad(i=1,2)
\end{equation}
for all $0<|k|<\sqrt{\gamma/|I_0^i|)}\sim \ve^{-1/2}$; i.e, if we start in $I_0^i$ it takes many steps (under rotation by $\omega$)
before we enter $I_0^i$ again. However, if we, for example, start in $I_ 0^1$, we may very well enter $I_0^2$ in just a few steps. 
We therefore need to check if the intervals $I_0^1,I_0^2$ are what we call "resonant". By this we 
mean that one of them is close to a translation (by few steps) of the other. 
We shall now fix different scales, depending on whether the intervals $I_0^1$ and $I_0^2$ are resonant or not.
We stress that we do not seek for optimal choices, in the sense that the $\ve_0(\gamma,\kappa,\rho)$ in Proposition \ref{pt_prop} is as small as possible.
Any such $\ve_0$ is enough for us.
 
\begin{figure}
\psfrag{i1}{$I_0^1$}
\psfrag{i2}{$I_0^2$}
\psfrag{i3}{$I_0^2-\nu_0\omega$}
\psfrag{j}{$J_0$}
\psfrag{x}{$x$}
\includegraphics[width=5cm]{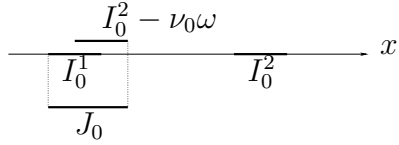}
\caption{Definition of $J_0$ in the resonant case.}\label{fig7}
\end{figure}
{\bf Case R}: If there is a $\nu_0$, $0<\nu_0\leq \ve^{-1/40}$ such that
$$
I_0^1\cap (I_0^2-\nu_0\omega)\neq \emptyset,
$$
(if $I_0^2\cap (I_0^1-\nu_0\omega)\neq \emptyset$ we relabel the intervals) let $J_0$ be the single interval
$$
J_0=I_0^1\cup (I_0^2-\nu_0\omega)
$$
and let 
$$
K_0=[\ve^{-1/20}] \text{ and } M_0=K_0^2=[\ve^{-1/20}]^2.
$$
By applying Lemma \ref{M_L2}, we see that
\begin{equation}\label{Ch_eq0}
\begin{aligned}
(J_0&+k\omega)\cap I_0=\emptyset \text{ for all } k\in[-M_0^2,M_0^2]\setminus\{0,\nu_0\}, \text{ and } \\
J_0&\cap I_0=I_0^1, ~(J_0+\nu_0\omega)\cap I_0=I_0^2.
\end{aligned}
\end{equation}
We note that $0<\nu_0^2\leq K_0$. We also note, by applying the same argument as in (\ref{Ch_eq00}), using the fact that $J_0$ is a single interval and 
$|3J_0|\sim \ve$, that 
\begin{equation}\label{Ch_eq01}
3J_0\cap (3J_0+\omega)=\emptyset \text{ for all } 0<|k|\leq M_0^2.
\end{equation}

{\bf Case NR}: Otherwise, i.e., if
$I_0\cap (I_0+k\omega)=\emptyset \text{ for all } 0<|k|\leq \ve^{-1/40}$,
let $J_0=I_0$ (thus $J_0$ consists of two intervals), and
$$
K_0=[\ve^{-1/160}] \text{ and } M_0=K_0^2=[\ve^{-1/160}]^2.
$$
We note that $M_0^2\leq \ve^{-1/40}$. We thus have
\begin{equation}\label{Ch_eq1}
J_0\cap (J_0+k\omega)=\emptyset \text{ for all } 0<|k|\leq M_0^2.
\end{equation}
Since each interval in $J_0$ is of length $\sim\ve$ we also have 
\begin{equation}\label{Ch_eq02}
3J_0^i\cap (3J_0^i+k\omega)=\emptyset \text{ for all } 0<|k|\leq M_0^2 \text{ and } i=1,2.
\end{equation}

\smallskip
For a (small) technical reason (at one place we would have liked  $\T\setminus J_0$ to be closed) we now define an open set $\widehat{J}_0\supset J_0$;
we just make $J_0$ slightly bigger. In Case R we let $\widehat{J}_0$ be a single open interval containing $J_0$, and so close to $J_0$ that (\ref{Ch_eq01}) 
holds with $J_0$ replaced by $\widehat{J_0}$ (since $J_0$ is closed, (\ref{Ch_eq01}) is an open condition). 
We also assume that $|\widehat{J}_0\setminus J_0|<\ve$. In Case NR we cover each interval $J_0^i$ with an open interval $\widehat{J}_0^i$, close to $J_0^i$,
such that (\ref{Ch_eq02}) holds with $3J_0^i$ replaced by $3\widehat{J}_0^i$, and such that (\ref{Ch_eq1}) holds with $J_0$ replaced by $\widehat{J}_0=\widehat{J}_0^1\cup\widehat{J}_0^2$.
Finally we assume that  $|\widehat{J}^i_0\setminus J_0^i|<\ve$ ($i=1,2$). 

\smallskip
The purpose of the $M_0$ and $K_0$ is the following. 
We want to focus on points $x\in\T$ which do not enter (under translation) the "bad" set $I_0$ too fast (for at least $M_0$ steps). If $y\in\T\setminus A$
we can then use Lemma \ref{FS_L1} to get a good control on the iterates; and $\prod_{k=1}^{N(x;I_0)}f'(y_j)$ is very small. We cannot avoid entering $I_0$, and
once we do that we lose control on the further iterates.  However, we will see that there is only for a very tiny set of $x$ (which we started with) 
for which we have not regained  control after $K_0$ steps from when we entered $I_0$; and once we have control again, we know that we will not enter the 
bad set $I_0$ for a long time ($\gg M_0$). The situation is complicated by the fact that $I_0$ consists of two intervals; thus we may not have recovered
from the first "hit" before we enter $I_0$ again. This is why we introduced the set $J_0$, in order to keep track on when we have "fast returns" (resonances) or not.

From the above definitions we see that $N(x;J_0)\leq N(x;I_0)$ for "many" $x$. Namely, 
\begin{lemma}\label{FS_L3}
If $x\in X_0=\T\setminus \bigcup_{m=-M_0+1}^{K_0-1}(J_0+m\omega)$, then $N(x;J_0)\leq N(x;I_0)$.
\end{lemma}
\begin{rem}Note that $|\bigcup_{m=-M_0+1}^{K_0-1}(J_0+m\omega)|\leq (M_0+K_0)|J_0|\ll\sqrt{\ve}$, so $X_0$ is a large set.
\end{rem}
\begin{proof}In the non-resonant case, when $I_0=J_0$, there is nothing to prove. Assume that we are in the resonant case, and let
$m=N(x;I_0)$. Since $J_0=I_0^1\cup(I_0^2-\nu_0\omega)$, and $0<\nu_0\ll K_0$, it follows from the assumption on $x$ that $m> \nu_0$.
If $x_m\in I_0^1\subset J_0$, then $N(x;J_0)\leq m$; if $x_m\in I_0^2\subset J_0+\nu_0\omega$, then $x_{m-\nu_0}\in J_0$, and thus
$N(x;J_0)\leq m-\nu_0$.

\end{proof}

\subsection{Geometry control}

Now we will send in a "probe" through $J_0$, which we later can follow ("most" iterates will cluster around the probe) in the transition from $J_0$ to $J_0+K_0\omega$.
The next lemma contains the geometric information about this probe. Obtaining control on the geometry, both in the base case as well as in the inductive step,
is the trickiest part in the proof of Proposition \ref{pt_prop}. However, the topological properties of the map $T$ helps quite a bit (some of the properties are formulated
in the lemmas in Section \ref{stf}).

Recall that the reference point  $\beta\in \T\setminus A$ was defined in (\ref{beta}). At many occasions in the paper we will use the crucial fact that if 
the functions $y_k(x)$ are defined by $y_k(x)=\pi(T^k(x,\beta))$, then, by Lemma \ref{SF_L3}(4), we have $y'_k(x)>1/\kappa$ on $\T$ for all $k\geq 1$. In particular,
the functions are strictly increasing.


\begin{lemma}\label{FS_L2}
Let $\Gamma_0=(J_0-M_0\omega)\times\{\beta\}$, and define $I_1\subset J_0$ by
$$
I_1=\pi_1(T^{M_0+K_0}(\Gamma_0)\cap(\T\times A'))-K_0\omega.
$$
Then  $I_1$ consists of two disjoint intervals, $I_1^1$ and $I_1^2$, and $\dist(I_1,\partial J_0)>$. 
Moreover, if we write
$T^{M_0+K_0}(\Gamma_0)=\{(x+K_0\omega,\vf_0(x)):x\in J_0\}$, then $\vf(I_1^i)=A' $ and $\ve^{-K_0/2}<\vf_0'(x)<\ve^{-\rho K_0}$ on $I_1^i$ $(i=1,2)$.
\end{lemma}

\begin{proof}
By the definition of $\Gamma_0$ we have
$$
\varphi_0(x)=\pi_2\left(T^{M_0+K_0}(x-M_0\omega,\beta)\right).
$$

We divide the analysis into two cases, depending on whether $J_0$ consists of one or two intervals (recall the definition of $J_0$ in the previous section). 
(If $g$ would have degree 1, and not $2$ as we have assumed, the set $I_0$ would only contain one single interval, and we would only have the first of the two cases below; 
no resonances would be possible.) 

\medskip
\begin{figure}
\psfrag{a}{$J_0^i-M_0\omega$}
\psfrag{b}{$J_0^i+\omega$}
\psfrag{c}{$J_0^i+K_0\omega$}
\psfrag{A}{$A$}
\psfrag{Ap}{$A'$}
\psfrag{i1}{$I_1^i+K_0\omega$}
\psfrag{be}{$\beta$}
\psfrag{g0}{$\Gamma_0^i$}
\psfrag{g1}{$T^{M_0+1}(\Gamma_0^i)$}
\psfrag{g2}{$T^{M_0+K_0}(\Gamma_0^i)$}
\psfrag{sh}{\tiny $T^{K_0-1}((J_0^i+\omega)\times\{\beta\})$ \normalsize}
\includegraphics[width=10cm]{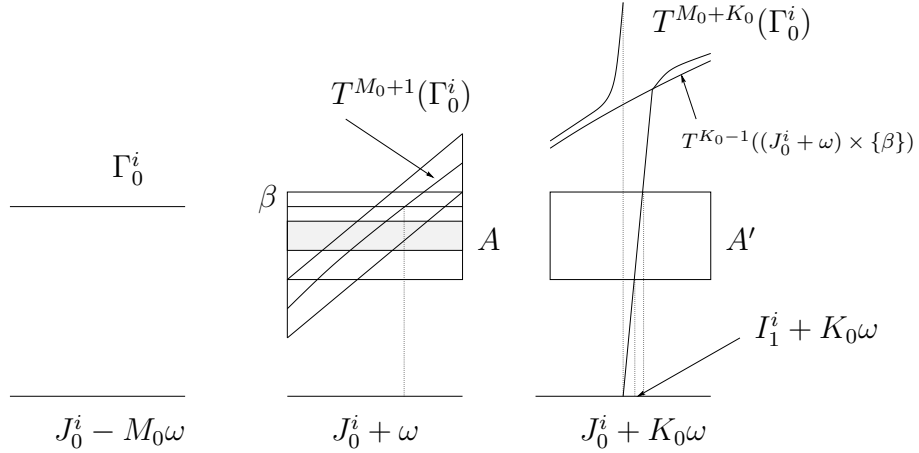}
\caption{The geometry in the non-resonant case ($i=1,2$).}\label{fig8}
\end{figure}
\noindent\emph{Case NR ($J_0=I_0$; the non-resonant case)}. We focus on one of the two intervals in $I_0$, call it $I_0^1$. The other interval is treated in exactly the same way. 
We shall write $I_0^1=[s,t]$.

Let
$$
y_k(x)=\pi_2\left(T^{k}(x-M_0\omega,\beta)\right).
$$
Then we can write $T^k(\Gamma_0)=\{(x+(k-M_0\omega),y_k(x)):x\in I_0\}$.
Thus, to describe the part of the set $I_1$ which lies in $I_0^1$, we need to control $y_{K_0+M_0}(x)$ for $x\in I_0^1$.
We note that if $\theta\in I_0-M_0\omega$, then  $N(\theta;I_0)=M_0$ (by (\ref{Ch_eq1})).  
Since $\beta\in \T\setminus A$, it therefore follows from Lemma \ref{FS_L1}(1) that
\begin{equation}\label{FS_L2_eq0}
(x+\omega,y_{M_0+1}(x))\in S, ~x\in I_0^1.
\end{equation}
Moreover, it also follows that for all $x\in I_{0}-M_0\omega$ we get
$$
\prod_{j=\ell}^N f'(\sigma_j)<\ve^{N-\ell+1} \text{ for all } 1\leq \ell \leq N \text{ and a.e. } \sigma\in \T\setminus A
$$
(here $\sigma_j=\pi_2(T^j(x,\sigma))$).
By applying Lemma \ref{SF_L3}(1), and using the trivial lower bound from Lemma \ref{SF_L3}(4) and the fact that $y_0$ is constant, we get 
\begin{equation}\label{FS_L2_eq01}
\frac{1}{\kappa}<y_{M_0+1}'(x)<2\kappa \text{ on } I_0^1.
\end{equation}
From Lemma \ref{SF_L3}(3) it therefore follows that we have the upper bound
$$
y_{M_0+K_0}'(x)<3\kappa\ve^{-\rho (K_0-1)}< \ve^{-\rho K_0} \text{ on } I_0^1.
$$
Combining (\ref{FS_L2_eq0}) with Lemma \ref{P_LB} we see that 
\begin{equation}\label{FS_L2_eq001}
\{x\in I_0^1: y_{M_0+1}(x)\in A'\} \text{ is a single interval, and } y_{M_0+1}(I_0^1)\supset A'. 
\end{equation}

Next we show that the part of $I_1$ in $I_0^1$ is well-inside $I_0^1$, namely that
$$
\dist(I_1\cap I_0^1,I_0^1)>\ve/(4\kappa). 
$$
To do this we note that $N(x;I_0)\gg K_0$ for $x\in I_0+\omega$, and since $(I_0+K_0\omega)\cap (I_0+\omega)=\emptyset$ (by (\ref{Ch_eq1})), it follows 
from Lemma \ref{FS_L1}(1) that $I_1\cap I_0^1\subset I_1':=\{x\in I_0^1: y_{M_0+1}(x)\in A'\}$.  (Indeed, if $(\theta,\eta)\in (I_0^1+\omega)\times (\T\setminus A)$, then, 
by Lemma \ref{FS_L1}(1), we get $\eta_{K_0-1}\notin A'$.)
Hence it follows from Lemma \ref{SF_L4} and the bound in (\ref{FS_L2_eq01}) that
$
\dist(I_1\cap I_0^1,I_0^1)>\ve/(4\kappa). 
$

To get control on $\varphi_0(x)=y_{M_0+K_0}(x)$ for $x\in I_0^1$, and thus control one of the two pieces of $\Gamma_0$, we will "shadow" 
the orbit of $(x+\omega,y_{M_0+1}(x))$, for $x\in I_0^1$, 
with another one which is easier to follow. To do this, we let
$$
\eta_k(x)=\pi_2\left(T^{k}(x+\omega,\beta)\right).
$$
By applying Lemma \ref{Tf_1} to the restriction of $y_{M_0+1}(x)$ to $I_0^1=[s,t]$, we see that (recall (\ref{FS_L2_eq001}) and note that $\eta_0=\beta$)
\begin{equation}\label{FS_L2_eq1}
[\eta_0,y_{M_0+1}(s)], [\eta_0,y_{M_0+1}(t)]\subset \T\setminus A,
\end{equation}
and there is a unique $\xi\in I_0^1$ such that $y_{M_0+1}(\xi)=\eta_0$, this $\xi\in \inn(I_0^1)$, and $y_{M_0+1}$ overtakes $\eta_0$ at $\xi$. See Fig. \ref{fig8}.
By Lemma \ref{Tf_0} we thus know that $y_{M_0+1+k}$ overtakes $\eta_k$ at $\xi$ for all $k$ (and $\xi$ obviously is the unique point of intersection in $I_0^1$).
Since $N(x+\omega;I_0)\gg K_0$ for all $x\in I_0$ (by (\ref{Ch_eq1})), and since $\beta\in \T\setminus A$, it follows from Lemma \ref{FS_L1}(1) that 
$$
\eta_{K_0-1}(x))\in \T\setminus A' \text{ for all } x\in I_0^1;
$$
and since (\ref{FS_L2_eq1}) holds, the same argument shows that (recall Lemma \ref{P_L0}(1))
$$
[\eta_{K_0-1}(s),y_{M_0+K_0}(s)],  [\eta_{K_0-1}(t),y_{M_0+K_0}(t)]\subset  \T\setminus A' . 
$$
We are thus in a situation where we can apply Lemma \ref{Tf_2}(1) to conclude that there is a 
single interval $I_1^1\subset I_0^1$ such that $I_1^1=\{x\in I_0^1: y_{M_0+K_0}(x)\in A'\}$, and $y_{M_0+K_0}(I_1^1)=A'$.

It remains to check that $y_{M_0+K_0}'(x)$ is large on $I_1^1$. To do this we note that,  
by (\ref{Ch_eq1}), we have $I_0+\omega\subset \T\setminus \bigcup_{j=0}^{K_0-1}(I_0-j\omega)$. It therefore follows from Lemma 
\ref{FS_L1}(2) that if $(\theta,\sigma)\in (I_0+\omega)\times \T$ are such that $\sigma_{K_0-2}\in A'$, then
$$
\prod_{j=0}^{K_0-2}f'(\sigma_j)>\ve^{-(K_0-1)}
$$
whenever the derivatives exist. Since $T^{M_0+1}(\Gamma_0)=\{(x+\omega),y_{M_0+1}(x)): x\in I_0\}$, it thus follows from Lemma \ref{SF_L3}(2), together with
the estimates on $y_{M_0+1}$ in (\ref{FS_L2_eq01}) and the definition of 
$I_1^1$, that $y_{M_0+K_0}'(x)>\ve^{-(K_0-1)}/\kappa\gg \ve^{-K_0/2}$ for a.e. $x\in I_1^1$.

\medskip
\noindent\emph{Case R ($J_0=I_0^1\cup(I_0^2-\nu_0\omega)$; resonant case)}. We shall write the single interval $J_0=[s,t]$. 
As in the previous case we let
$$
y_k(x)=\pi_2\left(T^{k}(x-M_0\omega,\beta)\right).
$$
From (\ref{Ch_eq0}) it follows that $N(x;I_0)\geq M_0$ for all $x\in J_0-M_0\omega$. Exactly as in the NR case we get
\begin{equation}\label{FS_L2_eq3}
(x+\omega,y_{M_0+1}(x))\in S \text{ and } \frac{1}{\kappa}<y_{M_0+1}'(x)<2\kappa , ~x\in J_0,
\end{equation}
as well as the upper bound
$$
y_{M_0+K_0}'(x)<\ve^{-\rho K_0} \text{ on } J_0.
$$ 
Since $J_0\cap I_0^2=\emptyset$ by (\ref{Ch_eq0}), it follows from Lemma \ref{P_LB} that $\{x\in J_0: y_{M_0+1}(x)\in A'\}$ is a single interval, and
$y_{M_0+1}(J_0)\supset A'$. 

\begin{figure}
\psfrag{i1}{$J_0+(\nu_0+1)\omega$}
\psfrag{i2}{$J_0+K_0\omega$}
\psfrag{b}{$\beta$}
\psfrag{g1}{\tiny $T^{M_0+\nu_0+1}(\Gamma_0)$\normalsize}
\psfrag{g2}{\tiny $T^{M_0+K_0}(\Gamma_0)$\normalsize}
\psfrag{p1}{\tiny $T^{\nu_0}((J_0+\omega)\times \{\beta\})$ \normalsize}
\psfrag{p2}{\tiny $T^{K_0-(\nu_0+1)}((J_0+(\nu_0+1)\omega)\times \{\beta\})$ \normalsize}
\psfrag{a}{$A'$}
\includegraphics[width=11cm]{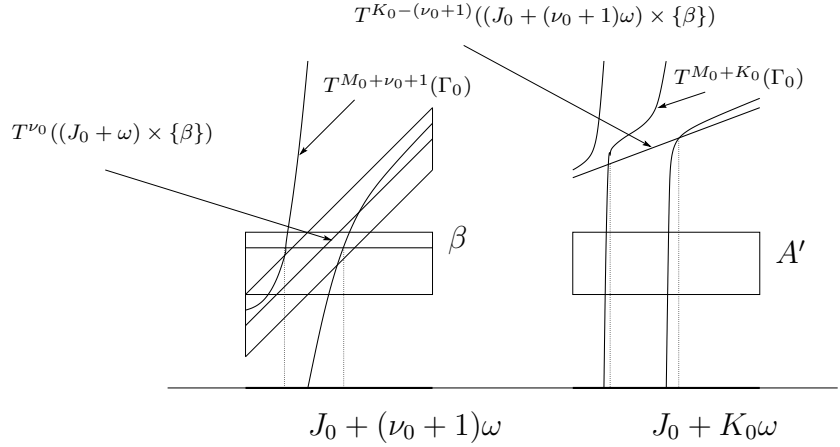}
\caption{The additions in the resonant case.}\label{fig9}
\end{figure}

In this resonant case we will use the shadowing argument twice: one time to get from $J_0+\omega$ to $(J_0+(\nu_0+1)\omega$, and a second
time to go from $(J_0+(\nu_0+1)\omega$ to $J_0+K_0\omega$. First, let, as above, $\eta_k(x)=\pi_2\left(T^{k}(x+\omega,\beta)\right)$. Applying Lemma \ref{Tf_1}
we get
\begin{equation}\label{FS_L2_eq2}
[\eta_0,y_{M_0+1}(s)], [\eta_0,y_{M_0+1}(t)]\subset \T\setminus A,
\end{equation}
and there is a unique $\xi\in J_0$ such that $y_{M_0+1}(\xi)=\beta=\eta_0$, and $y_{M_0+1}$ (trivially) overtakes $\eta_0$ at $\xi$.

By (\ref{Ch_eq0}) we have $(J_0+k\omega)\cap I_0=\emptyset$ for all $1\leq k\leq \nu_0-1$, and 
$(J_0+\nu_0\omega)\cap I_0=I_0^2$. We therefore get, by proceeding exactly as above (with $(x+\omega,\eta_0)$, $x\in J_0$), that
$$
(x+(\nu_0+1)\omega,\eta_{\nu_0}(x))\in S  , ~x\in J_0;
$$
and $\{x\in J_0: \eta_{\nu_0}(x)\in A'\}$ is a single interval, and
$\eta_{\nu_0}(J_0)\supset A'$. Moreover, since (\ref{FS_L2_eq2}) holds, we also get 
$$
\{s+(\nu_0+1)\omega\}\times [\eta_{\nu_0}(s),y_{M_0+\nu_0+1}(s)],  \{t+(\nu_0+1)\omega\}\times [\eta_{\nu_0}(t),y_{M_0+\nu_0+1}(t)]\subset S, 
$$
and since $s+(\nu_0+1)\omega, t+(\nu_0+1)\omega\ \notin \inn(I_0+\omega)$ we have (recall Lemma \ref{P_LA})
\begin{equation}\label{FS_L2_eq4}
[\eta_{\nu_0}(s),y_{M_0+\nu_0+1}(s)], [\eta_{\nu_0}(t),y_{M_0+\nu_0+1}(t)]\subset \overline{\T\setminus A'}.
\end{equation}
We are thus in a situation where we can apply Lemma \ref{Tf_3}. We conclude that the equation $y_{M_0+\nu_0+1}(x)=\beta$ has exactly 
two solutions $\xi_1,\xi_2$ in 
$J_0=[s,t]$, both lying in the interior of $J_0$, and $y_{M_0+\nu_0+1}(x)$ overtakes the constant function $\sigma_0(x)=\beta$ at these two points. See Fig. \ref{fig9}.
Moreover,
\begin{equation}\label{FS_L2_eq5}
[\beta,y_{M_0+\nu_0+1}(s)], [\beta,y_{M_0+\nu_0+1}(t)]\subset \T\setminus A.
\end{equation}

We can now use the above control to verify that the set $I_1$ must be "well-inside" $J_0$. We notice that,  
since (\ref{Ch_eq0}) holds, we get that $N(x,I_0)\gg K_0$ for $x\in J_0+(\nu_0+1)\omega$. Since $(J_0+K_0\omega)\cap I_0=\emptyset$, it thus follows (by the same argument as in the
previous case) from
Lemma \ref{FS_L1}(1) that $I_1\subset \{x\in J_0: y_{M_0+\nu_0+1}(x)\in A\}$. Since we have the derivative estimates in (\ref{FS_L2_eq3}), it follows from the trivial upper bound in
Lemma \ref{SF_L3}(3) that $y_{M_0+\nu_0+1}'<3\kappa\ve^{-\rho\nu_0}$ on $J_0$. Since also (\ref{FS_L2_eq4}) holds, it follows from Lemma \ref{SF_L4} that
$$
\dist(I_1,\partial J_0)>\ve^{\rho\nu_0+1}/(6\kappa).
$$

Finally, to get control on $y_{M_0+K_0}$ we apply shadowing a second time. We shall shadow $y_{M_0+\nu_0+1}$ by the iterates
$$
\sigma_k(x)=\pi_2\left(T^{k}(x+(\nu_0+1)\omega,\beta)\right).
$$
We noted above that $N(x,I_0)\gg K_0$ for $x\in J_0+(\nu_0+1)\omega$. 
Thus it follows from Lemma \ref{FS_L1}  (since $\beta\in\T\setminus A$) that
$$
\sigma_{K_0-(\nu_0+1)}(x)\in \T\setminus A' \text{ for all } x\in J_0;
$$
and since (\ref{FS_L2_eq5}) holds we get
$$
[\sigma_{K_0-(\nu_0+1)}(s),y_{M_0+K_0}(s)], [\sigma_{K_0-(\nu_0+1)}(t),y_{M_0+K_0}(t)] \subset \T\setminus A'. 
$$
Since $y_{M_0+K_0}$ overtakes $\sigma_{K_0-(\nu_0+1)}$ at $\xi_1,\xi_2$, we are thus in a situation where we can apply Lemma \ref{Tf_2}(2) 
and conclude that there are two disjoint intervals $I_1^1,I_1^2$ in $J_0$ such that
$I_1^1\cup I_1^2=\{x\in J_0: y_{M_0+K_0}(x)\in A'\}$; and $y_{M_0+K_0}(I_1^i)=A'$ ($i=1,2$). 

What is left now is to get a lower bound on $y_{M_0+K_0}'$ on $I_1$. Since $N(x,I_0)\gg K_0$ for all $x\in I_0+(\nu_0+1)\omega$ it follows 
from Lemma \ref{FS_L1}(2) that if $(\theta,\eta)\in (I_0+(\nu_0+1)\omega)\times \T$ are such that $\eta_{K_0-(\nu_0+1)}\in A'$ then
$$
\prod_{j=0}^{K_0-\nu_0-2}f'(\eta_j)>\ve^{-(K_0-\nu_0-1)}.
$$
Since we by Lemma \ref{SF_L3}(4) have the trivial bound $y_{M_0+\nu_0+1}'(x)>1/\kappa$ on $\T$, it thus follows from
Lemma \ref{SF_L3}(2) that $y_{M_0+K_0}'(x)>\ve^{-(K_0-\nu_0-1)}/\kappa \gg \ve^{-K_0/2}$ for a.e. $x\in I_1^1$ (recall that $\nu_0^2\leq K_0$).

\end{proof}
\begin{rem}As we also remarked earlier, assuming that $g(x)$ having degree $>2$ would not add any fundamental differences (we just got longer chains of the shadowing argument
when we have, the very unlikely, situations of multiple resonances). However, the notation would become slightly more involved. 
\end{rem}

\end{section}


\begin{section}{Inductive step}
We are now ready for the the inductive step in the proof of Proposition \ref{pt_prop}.
To simplify the statements in the next lemma we shall use the following notation: given sets $J_i$ and positive  integers $M_i,K_i$ ($i\geq 0$), we let 
$X_n, X'_n, G_n, H_n$ ($n\geq 0$) denote the following sets:
$$
X_n=\T\setminus \bigcup_{j=0}^n\bigcup_{m=-(M_j-1)}^{K_j-1}(J_j+m\omega), \quad X_{-1}=\T;
$$
$$
X_n'=\T\setminus \bigcup_{j=0}^n\bigcup_{m=-M_j^{3/2}}^{K_j-1}(J_j+m\omega), \quad X_{-1}'=\T;
$$

$$
G_n= \bigcup_{j=0}^n\bigcup_{m=1}^{K_j}(J_j+m\omega), \quad G_{-1}=\emptyset;\quad \text{and}
$$
$$
H_n=\T\setminus \bigcup_{j=1}^n\bigcup_{m=0}^{M_j}(J_j+m\omega), \quad H_{-1}=\T.
$$
We note that $X_n'\subset X_n$.

The (vague) motivation for these sets is that the sets $J_i$ should be thought of as "bad", and we need to control the approach rate to them (under translation by $\omega$),
as well as the recovering time after visits.

For consistency in the statement of the next lemma  we define 
$$
M_{-1}=K_{-1}=1.
$$
Moreover, in order to get the notations slightly clearer we shall make use of the following convention: Estimates of products of the form
$
\Pi_{j=k}^\ell f'(y_j),
$ 
should be thought to hold, under appropriate assumptions on $(x,y)$, whenever the expression exists (which it does for a.e $y\in \T$).

\begin{lemma}\label{ind_L1} 
There exists an $\ve_2(\gamma,\kappa,\rho)>0$ such that for all $0<\ve<\ve_2$ and any $n\geq 0$ we have the following:
\underline{Assume} that the closed sets $J_0\supset J_1\supset \cdots \supset J_n$ and the integers $M_j,K_j$ ($j=0,1,\ldots, n$)
have been constructed, and that the open sets $\widehat{J}_j\supset J_j$ ($j\in [0,n]$) have been chosen, so that the following four conditions hold: 

\medskip
\noindent Condition $(1)_n$: For each $j\in [0,n]$ there holds:  
\begin{equation}\label{ind_eq1}
1/2< M_j/K_j^2<2  \text{ and }  [\ve^{-K_{j-1}/160}]\leq K_j\leq 2[\ve^{-K_{j-1}/20}]; 
\end{equation}
\begin{equation}\label{ind_eq2}
\widehat{J}_j\cap (\widehat{J}_j+k\omega)=\emptyset \text{ for } 0<|k|\leq M_j^2;
\end{equation}
The sets $J_j$ and $\widehat{J}_j$ both consists of one or two intervals, $J_j^\ell, \widehat{J}_j^\ell$, and the intervals satisfy
\begin{equation}\label{ind_eq21}
\ve^{2\rho K_{j-1}}<|J_j^\ell|<|\widehat{J}_j^\ell|<\ve^{K_{j-1}/2}
\end{equation} 
and 
\begin{equation}\label{ind_eq3}
3\widehat{J}_j^\ell\cap (3\widehat{J}_j^\ell+k\omega)=\emptyset \text{ for } 0<|k|\leq M_j^2.
\end{equation}
Moreover, if $n\geq 1$, we have
\begin{equation}\label{ind_eq4}
J_n-M_n\omega, J_n+K_n\omega\subset \T\setminus \bigcup_{j=0}^{n-1}\bigcup_{m=-M_j^{3/2}}^{M_j^{3/2}}(\widehat{J}_j+m\omega) ~~\left(\subset X_{n-1}'\right).
\end{equation}

\noindent Condition $(2)_n$: If $(x,y)\in X_{n}\times (\T\setminus A)$ and $N=N(x;J_{n})$ then:
\begin{equation}\label{ind_eq5}
\prod_{j=1}^k f'(y_j)<\ve^{(1/2+1/2^{n+1})k} ~\text{ for all } k\in [1,N];
\end{equation}
for all $t\in [1,N]$ such that $x_t\in H_{n-1}$ we have
\begin{equation}\label{ind_eq6}
\prod_{j=k}^t  f'(y_j)<\ve^{(1/2+1/2^{n+1})(t-k+1)} \text{ for all } k\in [1,t]; \quad \text{and}
\end{equation}
\begin{equation}\label{ind_eq7}
\text{ if } y_k \in A' \text{ for some } k\in [1,N], \text{ then } x_k\in G_{n-1}.
\end{equation}

\noindent Condition $(3)_n$: 
Assume that $x\in X_n'$.  
If $(x_k,y_k)\in (\T\setminus G_{n-1})\times A'$ for some  $1\leq k\leq N(x;J_n)$ and some $y\in\T$, then
$$
\prod_{j=0}^{k-1} f'(y_j)>\ve^{-(1/2+1/2^{n+1})k}.
$$


\noindent Condition $(4)_n$: If $\Gamma_n=(J_n-M_n\omega)\times \{\beta\}$ and 
$$
I_{n+1}=\pi_1(T^{M_n+K_n}(\Gamma_n))\cap (\T\times A')-K_n\omega,
$$
then $I_{n+1}\subset J_n$ consists of two disjoint intervals $I_{n+1}^i$ $(i=1,2)$, and $\dist(I_{n+1},\partial J_n)>\ve^{K_n^{2/3}}$. Moreover, if we write 
$T^{M_n+K_n}(\Gamma_n)=\{(x+K_n\omega,\vf_{n}(x)): x\in J_n\}$, then $\vf_n(I_{n+1}^i)=A'$ and $\ve^{-K_n/2}<\vf_n'(x)<\ve^{-\rho K_n}$ on $I_{n+1}^i$ $(i=1,2)$.

\underline{Then} there is a closed set $J_{n+1}\subset J_n$, and an open set $\widehat{J}_{n+1}\supset J_{n+1}$, 
both consisting of one or two intervals, as well as integers $M_{n+1}, K_{n+1}$ such that
$(1-4)_{n+1}$ hold. Furthermore, the following statement holds for each interval $J_{n+1}^\ell$ in $J_{n+1}$: 

\noindent Condition $(C)_n$: If $\mathcal{O}\subset (J_{n+1}^\ell-M_n\omega)\times (\T\setminus A)$ is a set such that $\pi_1(\mathcal{O})$ is $\delta$-dense 
in $J_{n+1}^\ell-M_n\omega$, then $\pi_2(T^{M_n+K_n}(\mathcal{O}))$ is $(\ve^{-\rho K_n}\cdot \delta+2\ve^{(M_n+K_n)/2})$-dense in $A'$.

\end{lemma}

\begin{rem}We note that properties $(1-3)_n$ are of an arithmetic nature while $(4)_n$ contains information about the  geometry.
Moreover, condition $(3)_n$, which actually is a statement about iterations of $T^{-1}$, could have been stated in a more precise way, similar to $(2)_n$. But we do not really
need better estimates, so we have chosen to formulate the condition like this. (This is the same comment as the one in the remark following Proposition \ref{pt_prop}.) 
\end{rem}

Before proving this lemma we will show how it, together with the base case in the previous section, is used to prove Proposition \ref{pt_prop}.

\begin{proof}[Proof of Proposition \ref{pt_prop}]We recall that $M_{-1}=K_{-1}=1$. 
The sets $J_0,\widehat{J_0}$ and the integers $K_0,M_0$ were
defined in Subsection \ref{DJMK}. We now check that the premises $(1-4)_n$ in Lemma \ref{ind_L1} are satisfied for
$n=0$. 

From the definitions in Subsection \ref{DJMK} we note that condition $(1)_0$ holds (note that condition $(\ref{ind_eq4})$ is void in the case $n=0$). 
Condition $(2)_0$ easily follows from Lemma \ref{FS_L1}(1). Indeed, if $x\in X_0$ it follows from Lemma \ref{FS_L3} that $N(x;J_0)\leq N(x;I_0)$. Note that
$H_{-1}=G_{-1}=\emptyset$. (In particular, condition $(\ref{ind_eq7})_0$ says that if $y\in \T\setminus A$, then $y_k\in A'$ for all $1\leq k\leq N(x;J_0)$.)
Furthermore, condition $(3)_0$ follows from Lemma \ref{FS_L1}(2). Finally the geometric condition $(4)_0$ is the content of Lemma \ref{FS_L2}.

Thus the inductive machine starts and gives us closed sets $J_0\supset J_1\supset \ldots$,  open sets $\widehat{J}_j\supset J_j$ ($j\geq 0$), and 
integers $M_j,K_j$ ($j\geq 0$) such that conditions $(1-4)_n$ holds for all $n\geq 0$.

We begin by looking at the sizes of the integers $M_j,K_j$. If we let 
$\Delta=\ve^{-1/160}$, which can be made as big as we want by taking $\ve$ small, we have $K_{j}\geq \Delta^{K_{j-1}}$ for all $j\in[0,n]$.
By $(\ref{ind_eq1})$ we see that $M_j\approx K_j^2$.
We also note that if we let $N_j=M_j^2$, then 
\begin{equation}\label{ind_eq7p}
M_j/N_j< 2/K_j^2 \quad\text{ and } \quad M_j^{3/2}/N_j<2/K_j.
\end{equation}
Moreover, 
$$
K_j/M_j<2/K_j \text{ and } M_j/M_j^{3/2}<2/K_j.
$$

We now define $X\subset \T$ to be the set
$$
X=\T\setminus \bigcup_{j=0}^\infty\bigcup_{m=-M_j^{3/2}}^{K_j}(J_j+m\omega). 
$$ 
From $(1)_n$ we see that $M_j<2K_j^2<8\ve^{-K_{j-1}/10}$ and $|J_j|<2\ve^{K_{j-1}/2}$ for all $j\geq 0$. Thus
$$
|X|\geq 1-\sum_{j=0}^\infty2M_j^{3/2}|J_j|\approx 1.
$$

By definition we have $X_n\supset X_n'\supset X$  and $\T\setminus G_n\supset X$ for all $n\geq 0$.  We also note that if we fix $x\in X$, then $N(x;J_n)\geq M_n^{3/2}$ for all $n\geq 0$.
Therefore statement (1) in Proposition \ref{pt_prop}, and the first half of statement (2), follow immediately from $(2)_n$ ($n\geq 0$) in Lemma \ref{ind_L1}.
Moreover, statement (3) in Proposition \ref{pt_prop} follows from $(3)_n$ ($n\geq 0$) in Lemma \ref{ind_L1}

It remains to prove the second part of (2), i.e., to prove that 
$$
\overline{\{(x_j,y_j)\}_{j\geq 0}}=\T^2 \text{ for all } (x,y)\in X\times (\mathbb{T}\setminus A). 
$$
To do this, fix $(x,y)\in X\times (\mathbb{T}\setminus A)$, and let $\mathcal{O}=\{(x_j,y_j)\}_{j\geq 0}$ be the forward orbit of $(x,y)$. We recall that $A'\supset A$.
From $(\ref{ind_eq7})_n$ in $(2)_n$ it follows that if $x_j\in A'$ for some $j\geq 1$, then
 $x_j\in G_\infty=\cup_{n\geq 0}G_n$. Moreover, combining $(\ref{ind_eq2})_n$ and $(\ref{ind_eq4})_n$ gives us 
$$
(J_{n}-M_n\omega)\cap G_n=\emptyset.
$$ 
Thus, if $x_j\in J_{n}-M_n\omega$ and $y_j\in A'$, then $x_j\in  G_\infty\setminus G_{n}= \bigcup_{j=n+1}^\infty\bigcup_{m=0}^{K_j}(J_j+m\omega)$.
From the estimates in $(1)_n$ it follows that $|G_\infty\setminus G_{n}|\ll\ve^{K_n/4}.$ Moreover, by $(\ref{ind_eq21})_n$ 
each interval in $J_n$ has a length $>\ve^{2\rho K_{n-1}}\gg \ve^{K_n/4}$.
We conclude that for each $n\geq 0$ the set $\pi_1(\mathcal{O}\cap ((J_n-M_n\omega)\times (\mathbb{T}\setminus A)))$ is $\ve^{K_n/4}$- dense in $J_n-M_n\omega$. 
Hence, for each $n\geq 1$ and each $J_{n}^\ell$ it follows from condition $(C)_{n-1}$ that the set
\begin{equation}\label{eq_proof_pt_prop2}
\pi_2(\mathcal{O}\cap ((J_n^\ell+K_n\omega)\times A')) \text{ is } (\ve^{-\rho K_{n-1}}\cdot\ve^{K_n/4} +2\ve^{(M_{n-1}+K_{n-1})/2}) \text{-dense in } A'.
\end{equation}
We note that $\ve^{-\rho K_{n-1}}\cdot\ve^{K_n/4} +2\ve^{(M_{n-1}+K_{n-1})/2}\to 0$ as $n\to\infty$.

By compactness of $\T$ there must be a point $x^*\in \T$ and an infinite subsequence $n_i$ such that $\dist(x^*,J_{n_i}+K_{n_i}\omega)\to 0$ 
as $i\to\infty$. From (\ref{eq_proof_pt_prop2}) it thus follows that
\begin{equation}\label{eq_proof_pt_prop}
\overline{\mathcal{O}}\supset \{x^*\}\times A'.
\end{equation}
We claim that $x^*\in X_n$ for all $n\geq 0$. (This is the only place where we need the open intervals $\widehat{J}_n$.) 
Indeed, by $(\ref{ind_eq4})_n$ we have that each $J_{n}+K_n\omega\subset S_n:=
\T\setminus \bigcup_{j=0}^{n-1}\bigcup_{m=-2M_j}^{2M_j}(\widehat{J}_j+m\omega)$. We note that each $S_n$ is closed, and clearly $S_{n+1}\subset S_n$. Moreover, the estimates in $(1)_n$ show
(as we did above for $X$) that $|\cap_{n\geq0} S_n|\approx 1$. From Lemma \ref{M_L4} it therefore follows 
that $x^*\in \cap_{n\geq1} S_n\subset \cap_{n\geq0} X_n$.

To complete the proof we note that (since $x^*\in\cap_{n\geq0} X_n$) $N(x^*,J_n)\geq M_n$ for all $n\geq 0$. Thus it follows from $(2)_n$ (applied with
$x=x^*$) that  
$$
\prod_{j=1}^k f'(y_j)<\ve^{k/2} ~\text{ for all } k \geq 1 \text{ and  for a.e. } y\in\T\setminus A.
$$
Combining this estimate with Lemma \ref{SF_L1} yields 
$|\pi_1(T^k(x^*,\T\setminus A))|\to 0 $ as $k\to\infty$. Thus we must have $|\pi_1(T^k(x^*,A))|\to 1 $ as $k\to\infty$, and hence also
$\overline{\bigcup_{k\geq0}T^k(x^*,A)}=\T^2$. Since 
(\ref{eq_proof_pt_prop}) holds, we can therefore conclude that
$\overline{\mathcal{O}}=\T^2$.
\end{proof}

Now we turn to the proof of the inductive lemma, Lemma \ref{ind_L1}.

\begin{proof}[Proof of Lemma \ref{ind_L1}]
The proof, which is rather long, is divided into steps. The first thing we do is to make a few observation on the position of translates of the set $J_n$. 
From  $(\ref{ind_eq2})_n$ and $(\ref{ind_eq4})_n$, together with the definition of the set $X_n$, it immediately follows that
\begin{equation}\label{ind_eq8}
J_n-M_n\omega\subset X_n \text{ and } J_n+K_n\omega\subset X_n'\subset X_n.
\end{equation}
Furthermore, since $J_{j+1}\subset J_j$ ($j=0,1,\ldots, n-1$) and since (\ref{ind_eq2}) holds, we have
\begin{equation}\label{ind_eq8p}
J_n, J_n-M_n\omega \subset H_n\subset \T\setminus G_n.
\end{equation}
These facts will be used several times below. We also recall the discussion about the sizes of the integers $K_j,M_j$, and their relationships, from the proof of Proposition 
\ref{pt_prop} above. 

\medskip
\noindent\emph{Step 1} (Construction of $J_{n+1}$ and $\widehat{J}_{n+1}$, and choice of $M_{n+1},K_{n+1}$; verification of $(1)_{n+1}$): 
From the estimates in $(4)_n$ it follows that  each of the two intervals in $I_{n+1}$ has a length in the interval $(\ve^{2\rho K_n},\ve^{K_n/2}/2)$.
Thus these intervals are much much shorter than the ones in $J_j$ and $\widehat{J}_j$ ($j=0, \ldots, n$), as given by the estimates in $(\ref{ind_eq21})_j$.

When we shall choose $M_{n+1}$ and $K_{n+1}$ we will use Lemma \ref{M_L3} (with the $Z_j=M_j^{3/2}$). From the above observation on the length of the intervals, together with
(\ref{ind_eq7p}) and $(\ref{ind_eq3})_j$ $(j=0,\ldots, n)$, we see that the premises for Lemma \ref{M_L3} are fulfilled for any interval
$J$ of length $<\ve^{K_n/2}$. (Here the $p$ in Lemma \ref{M_L3} is $2$).

We now check whether the two intervals in $I_{n+1}$ in 
$(4)_n$ are resonant or not. We know that the length of each of the intervals satisfies $<\frac12\ve^{K_n/2}$. Thus it follows from Lemma \ref{M_L1}
that 
$$
I_{n+1}^\ell\cap(I_{n+1}^\ell+k\omega)=\emptyset \text{ for all } 0<|k|\leq \sqrt{\gamma}\cdot\ve^{-K_n/4}.
$$

Case R (resonant case): If there is a $\nu_{n+1}$, $0<\nu_{n+1}\leq\ve^{-K_n/40}$, such that
$$
I_{n+1}^1\cap (I_{n+1}^2-\nu_{n+1}\omega)\neq \emptyset
$$
then we let $J_{n+1}$ be the interval $J_{n+1}=I_{n+1}^1\cup (I_{n+1}^2-\nu_{n+1}\omega)$ 
(if the above relation holds with $I_{n+1}^1$ and $I_{n+1}^2$ interchanged, we relabel the intervals). 
Thus the length of $J_{n+1}$ is smaller than $\ve^{-K_n/2}$. 
Let $K=[\ve^{-K_n/20}]$ and $M=K^2$. Applying Lemma \ref{M_L3} with $J=J_{n+1}+K\omega$ shows that there is a $K_{n+1}\in [K+1,K+M_n^2]$ such that
$J_{n+1}+K_{n+1}\omega$ satisfy $(\ref{ind_eq4})_{n+1}$; and applying  Lemma \ref{M_L3} with $J=J_{n+1}-M\omega$ gives us an $M_{n+1}\in [M-M_n^2,M-1]$ such that
$J_{n+1}+M_{n+1}\omega$ satisfy $(\ref{ind_eq4})_{n+1}$. Since $M_n^2\approx K_n^4\ll K$ (by $(\ref{ind_eq1})_n$), it follows that $M_{n+1}/K_{n+1}\approx M/K^2=1$.
Thus $(\ref{ind_eq1})_{n+1}$ holds. Finally, we have $|J_{n+1}|<\ve^{-K_n/2}$ (and $J_{n+1}$ is closed). Let $\widehat{J}_{n+1}\supset J_{n+1}$ be an open interval
satisfying $|\widehat{J}_{n+1}|<\ve^{-K_n/2}$ (any such interval works). Since $M_{n+1}^2\approx \ve^{-K_n/5}\ll \ve^{-K_n/4}$,
conditions $(\ref{ind_eq2}-\ref{ind_eq3})_{n+1}$ immediately follow from Lemma \ref{M_L1}. Furthermore, from Lemma \ref{M_L2} we get
\begin{equation}\label{ind_eq90}
\begin{aligned}
(J_{n+1}&+k\omega)\cap I_{n+1}=\emptyset \text{ for all } k\in[-M_{n+1}^2,M_{n+1}^2]\setminus\{0,\nu_{n+1}\}, \text{ and } \\
J_{n+1}&\cap I_{n+1}=I_{n+1}^1, ~(J_{n+1}+\nu_{n+1}\omega)\cap I_{n+1}=I_{n+1}^2.
\end{aligned}
\end{equation}

From the above definitions we see that $\nu_{n+1}^2\leq K_{n+1}$. Moreover, since $(\ref{ind_eq2})_n$ holds, and $I_{n+1}\subset J_n$, we have $\nu_{n+1}\geq M_n^2$. Thus,
\begin{equation}\label{ind_eq91}
M_n^2\leq \nu_{n+1}\leq K_{n+1}^{1/2}.
\end{equation}
Furthermore , since $I_{n+1}\subset J_n$, and since $\dist(I_{n+1},J_n)>\ve^{K_n^{2/3}}\gg 
\ve^{K_n/2}$, it follows that 
\begin{equation}\label{ind_eq92}
J_{n+1}=I_{n+1}^1\cup (I_{n+1}^2-\nu_{n+1}\omega)\subset J_{n} \text{ and } J_{n+1}+\nu_{n+1}\omega\subset J_n
\end{equation} 
(since $|I_{n+1}^\ell|<\ve^{K_n/2}$).

Case NR (non-resonant case): Otherwise we have
\begin{equation}\label{ind_eq9p}
I_{n+1}\cap (I_{n+1}+k\omega)=\emptyset \text{ for all } 0<|k|\leq \ve^{-K_n/40}.
\end{equation}
In this case we let $J_{n+1}=I_{n+1}$; thus $J_{n+1}$ contains two intervals (and trivially we have $J_{n+1}\subset J_n$). We let $\widehat{J}_{n+1}\supset J_{n+1}$ be an open set 
consisting of two intervals, each of length $<\ve^{K_n/2}/2$, and such that (\ref{ind_eq9p}) holds with $I_{n+1}$ replaced by $\widehat{J}_{n+1}$ (this is possible since $I_{n+1}$ is closed).
Next, let $K=[\ve^{-K_n/160}]$ and $M=K^2$. 
By proceeding as above, applying Lemma \ref{M_L3} with
$J=J_{n+1}+K\omega$ and $J=J_{n+1}+M\omega$, respectively, gives us $M_{n+1}\in [M-M_n^2,M-1]$ and $K_{n+1}\in [K+1,K+M_n^2]$ such that
$(\ref{ind_eq4})_{n+1}$ holds. Moreover, since $M_{n+1}^2\leq [\ve^{-K_n/40}]$, condition $(\ref{ind_eq2})_{n+1}$ follows from (\ref{ind_eq9p}); and $(\ref{ind_eq3})_{n+1}$
follows from Lemma \ref{M_L1} (since $|\widehat{J}_{n+1}^\ell|<\ve^{-K_n/2}/2$).

\medskip
From these definitions we get, in particular:
\begin{sublem}\label{ind_sl0}
If $x\in \T\setminus \bigcup_{m=0}^{K_{n+1}-1}(J_{n+1}+m\omega) \subset X_{n+1}$, then $N(x; J_{n+1})\leq N(x;I_{n+1})$.
\end{sublem}
\begin{proof}Since $\nu_{n+1}\ll K_{n+1}$, the proof is almost identical to that of Lemma \ref{FS_L3}. 
\end{proof}

\medskip
\noindent\emph{Step 2} (Verifying $(C)_n$ and $(2)_{n+1}$): First we note that $(1-2)_n$ gives us the following result. 
\begin{sublem}\label{ind_sl1} 
If $x\in J_n-M_n\omega$ then 
$$
\prod_{j=1}^{k}f'(y_j)<\ve^{(1/2+1/2^{n+2})k}
$$
for all $k\in[0,M_n+K_n]$ and a.e. $y\in \T\setminus A$ .
\end{sublem}
\begin{proof}
From  $(\ref{ind_eq8})$ we have $x\in J_n-M_n\omega\subset X_n$; thus $N(x;J_n)=M_n$.
Hence we can apply $(2)_n$ and conclude that $
\prod_{j=1}^{k}f'(y_j)<\ve^{(1/2+1/2^{n+1})k}
$
for all $k\in[0,M_n]$ and a.e. $y\in \T\setminus A$. If $k\in [M_n+1,M_n+K_n]$ we therefore, trivially, get (since $f'(\eta)<\ve^{-\rho}$ for a.e. $\eta\in \T$) that
$$
\prod_{j=1}^{k}f'(y_j)<\ve^{(1/2+1/2^{n+1})M_n - \rho(k-M_n)}
$$
for a.e. $y\in \T\setminus A$. Thus, if we show that $(1/2+1/2^{n+1})M_n - \rho(k-M_n)>(1/2+1/2^{n+2})k$ we are done. 

This inequality can be written $(\rho+1/2)M_n+M_n/(2^{n+1})-(\rho+1/2)k-k/(2^{n+2})>0$. Thus the worst case is when $k=M_n+K_n$. For this $k$ the inequality becomes
$$
\frac{M_n}{2^{n+2}}-(\rho+1/2)K_n-\frac{K_n}{2^{n+2}}>0.
$$
Since $\rho+1/2+1/2^{n+2}<2\rho$, this inequality is fulfilled if $M_n>2^{n+2}\cdot 2\rho K_n$ (or $K_n/M_n<1/(2^{n+3}\rho)$).
But $K_n/M_n<2/K_n\ll 1/(2^{n+3}\rho) $, so indeed the statement of the lemma holds.  

\end{proof}
Combining the previous sublemma with Lemma \ref{SF_L1} immediately gives us
\begin{sublem}\label{ind_sl2}
 If $x\in J_n-M_n\omega$ and $y,\eta\in\T\setminus A$, then
$$d(y_{M_n+K_n},\eta_{M_n+K_n})\leq \ve^{(1/2+1/2^{n+2}){(M_n+K_n)}}d(y,\eta).$$ 
\end{sublem}
This last sublemma enables us to verify condition $(C)_n$. Assume that $(x,y),(\xi,\eta)\in (J_{n}^\ell-M_n\omega)\times \T\setminus A$ and that 
$d(x,\xi)<\delta$. From the upper bound on the derivative of $\varphi_n$ in condition $(4)_n$ it follows that 
$d(\pi_2(T^{M_n+K_n}(x,\beta)),\pi_2(T^{M_n+K_n}(\xi,\beta)))<\ve^{-\rho K_n}\cdot \delta$. 
Moreover, since $\beta\in\T\setminus A$ we can apply Sublemma \ref{ind_sl2} to get
$d(\pi_2(T^{M_n+K_n}(x,\beta)),\pi_2(T^{M_n+K_n}(x,y)))<\ve^{(M_n+K_n)/2}$ and  $d(\pi_2(T^{M_n+K_n}(\xi,\beta),\pi_2(T^{M_n+K_n}(\xi,\eta))<\ve^{(M_n+K_n)/2}$. We can thus conclude that
$$d(\pi_2(T^{M_n+K_n}(x,y),\pi_2(T^{M_n+K_n}(\xi,\eta))< \ve^{-\rho K_n}\cdot \delta+2\ve^{(M_n+K_n)/2}.$$ This shows that $(C)_n$ holds.

\smallskip
The way to prove $(2)_{n+1}$ is to use the following sublemma. In the proof of this sublemma the "probe" in statement $(4)_{n}$ will be used.
\begin{sublem}\label{ind_sl4}
If $(x,y)\in X_n\times (\T\setminus A)$ and $N=N(x;I_{n+1})$, then the statements 
$(\ref{ind_eq5}-\ref{ind_eq7})_{n+1}$ hold.
\end{sublem}

Note that $(2)_{n+1}$ indeed follows from this, since $X_{n+1}\subset X_n$, and since $N(x;J_{n+1})\leq N(x;I_{n+1})$ for all $x\in X_{n+1}$ by Sublemma \ref{ind_sl0}.

\begin{proof}[Proof of Sublemma \ref{ind_sl4}]
Fix $x\in X_n$ and let $N=N(x;I_{n+1})$. For simplicity we also fix $y\in \T\setminus A$ such that $f'(y_j)$ exists for all $j\geq 0$ (recall that this holds for
a.e. $y$).

The plan is to use $(2)_n$ repeatedly. We let $0<t_0< t_1< \ldots< t_s=N$ be the times $\leq N$ when $x_k\in J_{n}$ (recall that $I_{n+1}\subset J_n$).
First we shall use $(2)_n$ to get control up to time $t_0$. Then we use the probe from $(4)_n$ to get control on the iterate $(x_{t_0+K_n},y_{t_0+K_n})$ (we allow to
lose control under the passage from time $t_0$ to time $t_0+K_n$). Now we can again (as we will see) use $(2)_n$ (applied to $(x_{t_0+K_n},y_{t_0+K_n})$) to
get information up to time $t_1$; and then we use the probe to get control on $(x_{t_1+K_n},y_{t_1+K_n})$. Continuing like this we get control up to time $t_s=N$.

Now we show the details in this approach. As above we let $0<t_0< t_1< \ldots< t_s=N$ be the times $\leq N$ when $x_k\in J_{n}$. Note that
$t_0\geq M_n$ since $x\in X_n$, and $t_{j+1}-t_j> M_n^2$ by $(\ref{ind_eq2})_n$.
  
Applying $(2)_n$ to the point $(x,y)$ we get that $(\ref{ind_eq5}-\ref{ind_eq7})_{n}$ hold, with 
$N$ replaced by $t_0$. We denote this by $(\ref{ind_eq5}-\ref{ind_eq7})_{n}[t_0]$. If $t_0=N$, i.e., if $x_{t_0}\in I_{n+1}$, the (weaker) statements 
in $(2)_{n+1}$ follow directly from $(2)_n$. If not, we note that
$x_{t_0-M_n}\in J_n-M_n\omega$. Since $J_n-M_n\omega\subset \T\setminus G_{n-1}$, by $(\ref{ind_eq4})_n$, it follows from
$(\ref{ind_eq7})_n[t_0]$ that $y_{t_0-M_n}\in \T\setminus A$. 
Combining $(\ref{ind_eq5})_n[t_0]$ with Sublemma \ref{ind_sl1} gives  
$(\ref{ind_eq5})_{n+1}[t_0+K_n]$. 

Next we recall the definition of $\Gamma_n$ and the function $\varphi_n$ in condition $(4)_n$.
We let $\eta=\beta$, and we shall compare the iterates of the two points $(x_{t_0-M_n},y_{t_0-M_n})$ and $(x_{t_0-M_n},\eta)$, which both are in 
$(J_n-M_n\omega)\times (\T\setminus A)$.
From the definition of $I_{n+1}$ (we have assumed that $x_{t_0}\notin I_{n+1}$) it follows that $\eta_{M_n+K_n}\notin A'$. Moreover, from Sublemma
\ref{ind_sl2} we get that $d(\eta_{M_n+K_n},y_{t_0+K_n})\ll \ve$. Recalling the definition of $A'$ in (\ref{ap}), we conclude that
$y_{t_0+K_n}\in \T\setminus A$, and thus 
\begin{equation}\label{ind_eq9}
(x_{t_0+K_n},y_{t_0+K_0})\in (J_n+K_n\omega)\times (\T\setminus A)\subset X_{n}\times (\T\setminus A),
\end{equation}
by $(\ref{ind_eq8})_n$.  We note that $(\ref{ind_eq7})_{n+1}[t_0+K_n]$ trivially follows from $(\ref{ind_eq7})_n[t_0]$ since 
$\bigcup_{m=1}^{K_n}(J_n+m\omega)\subset G_n$. Thus, so far we know that $(\ref{ind_eq5})_{n+1}[t_0+K_0]$ and $(\ref{ind_eq7})_{n+1}[t_0+K_n]$ hold. 
Moreover,  $(\ref{ind_eq6})_{n+1}[t_0+K_n]$ immediately follows from 
$(\ref{ind_eq6})_{n}[t_0]$, because of the definition of $H_n$. 

Since (\ref{ind_eq9}) holds, we can proceed exactly as above (taking one more turn, from time $t_0+K_n$ to time $t_1>t_0+M_n^2$, by applying $(2)_n$ 
to the point $(x_{t_0+K_n},y_{t_0+K_n})$ ). We get (from $(\ref{ind_eq5}-\ref{ind_eq7})_n$):
\begin{equation}\label{ind_eq10}
\prod_{j=t_0+K_n+1}^k f'(y_j)<\ve^{(1/2+1/2^{n+1})(k-t_0-K_n)}, k\in [t_0+K_n+1,t_1];
\end{equation}
For all $t\in [t_0+K_n+1,t_1]$ such that $x_t\in H_{n-1}$ we have:
\begin{equation}\label{ind_eq11}
\prod_{j=k}^t f'(y_j)<\ve^{(1/2+1/2^{n+1})(t-k+1)}, k\in [t_0+K_n+1,t];
\end{equation}
and 
\begin{equation}\label{ind_eq12}
\text{ if } y_k \in A' \text{ for some } k\in [t_0+K_n+1,t_1],  \text{ then } x_k\in G_{n-1}.
\end{equation}
Combining $(\ref{ind_eq5})_{n+1}[t_0+K_n]$ with (\ref{ind_eq10}) gives $(\ref{ind_eq5})_{n+1}[t_1]$; and combining 
$(\ref{ind_eq7})_{n+1}[t_0+K_n]$ with (\ref{ind_eq12}) gives $(\ref{ind_eq7})_{n+1}[t_1]$.

It just remains to check $(\ref{ind_eq6})_{n+1}[t_1]$. Take $t\in [t_0+K_n+1,t_1]$ such that $x_t\in H_n\subset H_{n-1}$. Then
$t>t_0+M_n$. We know that (\ref{ind_eq11}) holds,  
and from $(\ref{ind_eq6})_{n+1}[t_0+K_n]$ we have, since (\ref{ind_eq8p}) holds (i.e., $x_{t_0}\in H_{n}$),
$$
\prod_{j=k}^{t_0} f'(y_j)<\ve^{(1/2+1/2^{n+2})(t_0+1-k)}, k\in [0,t_0].
$$
Thus, to verify $(\ref{ind_eq6})_{n+1}[t_0+t_1]$ we need to show that 
$$
\prod_{j=k}^t f'(y_j)<\ve^{(1/2+1/2^{n+2})(t-k+1)}, k\in [t_0+1,t_0+K_n].
$$
Take $k\in [t_0+1,t_0+K_n]$. Then, by (\ref{ind_eq11}), and using that $f'(\eta)<\ve^{-\rho}$ for a.e. $\eta\in \T$, we get
$$
\prod_{j=k}^t f'(y_j)=\prod_{j=k}^{t_0+K_n} f'(y_j)\prod_{j=t_0+K_n+1}^t f'(y_j)<\ve^{-\rho(t_0+K_n-k+1)}\cdot\ve^{(1/2+1/2^{n+1})(t-t_0-K_n)}.
$$
Recalling that $t>t_0+M_n$, and by proceeding as in the proof of Sublemma \ref{ind_sl1} (again using the fact that $K_n^2\approx M_n$) shows that this last expression is 
smaller than $\ve^{(1/2+1/2^{n+2})(t-k+1)}$. Thus $(\ref{ind_eq6})_{n+1}[t_1]$ holds.

In this way we can continue 
inductively to show that $(\ref{ind_eq5}-\ref{ind_eq7})_{n+1}[t_j]$ hold for $j=1,2,\ldots, s$. The process stops when we come to time $t_s=N$ (since we then are in $I_{n+1}$ 
and do not know what happens with the further iterates). 

\end{proof}

\noindent\emph{Step 3} (Verifying $(3)_{n+1}$): Condition $(3)_{n+1}$ follows directly (by the same argument as the one after the statement of Sublemma \ref{ind_sl4}; 
recall that $X_n'\subset X_n$) from   
\begin{sublem}\label{ind_sl5}
Assume that $(x,y)\in X_n'\times\T$ and that for some $1\leq k\leq N=N(x; I_{n+1})$ we have 
$(x_k,y_k)\in (\T\setminus G_{n})\times A'$. Then
$$
\prod_{j=0}^{k-1} f'(y_j)>\ve^{-(1/2+1/2^{n+2})k}.
$$

\end{sublem}
\begin{proof}
Fix $x\in X_n'$, and let $0<t_0<t_1<\ldots< t_s=N=N(x;I_{n+1})$ be the times $\leq N$ when $x_k\in J_{n}$.  Be proceeding as in the proof of  Sublemma \ref{ind_sl4}, we note that
we have
$t_0> M_n^{3/2}$ and $t_{j+1}-t_j>M_n^2$. 

Let $(3)_{n+1}[t_j]$ denote the statement in $(3)_{n+1}$ where $N(x;J_{n+1})$ is replaced by $t_j$.
We thus have to show that $(3)_{n+1}[t_s]$ holds. Since $\T\setminus G_n\subset \T\setminus G_{n-1}$ we note that $(3)_{n+1}[t_0]$ follows directly from $(3)_n$ 
(which we have assumed holds). 

Next we shall prove that
$(3)_{n+1}[t_1]$ holds. Assume that $(x_k,y_k)\in (\T\setminus G_n)\times A'$ for some $t_0+1\leq k \leq t_1$ and some $y$. As we noted above we have $t_0> M_{n}^{3/2}$.
Moreover, since $x_k\notin G_n$, and $k\geq t_0+1$ (and $x_{t_0}\in J_n$), we must by the definition of $G_n$ have $k>t_0+K_n$. 

We claim that 
$$y_{t_0-M_n}\in A \text{ and } y_{t_0+K_n}\in A.$$
Indeed, if this was not the case we could apply Sublemma \ref{ind_sl4}  
(property $(\ref{ind_eq7})_{n+1}$) to the points $(x_{t_0-M_n}, y_{t_0-M_n})$ and $(x_{t_0+K_n}, y_{t_0+K_n})$, respectively (recall that (\ref{ind_eq8}) holds),
and conclude that $y_k\notin A'$ (since $x_k\in\T\setminus G_n$). 

The product $\prod_{j=0}^{k-1} f'(y_j)$ is now split into three parts:
$$
\prod_{j=0}^{k-1} f'(y_j)=\prod_{j=0}^{t_0-M_n-1} \prod_{j=t_0-M_n}^{t_0+K_n-1} \prod_{j=t_0+K_n}^{k-1}
$$
Since $y_{t_0-M_n}\in A$ (and since $(J_n-M_n\omega)\cap G_n=\emptyset$), it follows from $(3)_n$, applied to the point $(x,y)$ 
that
$
\prod_{j=0}^{t_0-M_n-1}f'(y_j)>\ve^{-(1/2+1/2^{n+1})(t_0-M_n)}.
$
Moreover, applying $(3)_n$ to the point $(x_{t_0+K_n},y_{t_0+K_n})$ (recall that (\ref{ind_eq8}) holds) gives
$
\prod_{j=t_0+K_n}^{k-1}f'(y_j)>\ve^{-(1/2+1/2^{n+1})(k-(t_0+K_n))}.
$
Since we have the global bound $f'(\eta)>\ve^{\rho}$, we thus get
$$
\prod_{j=0}^{k-1} f'(y_j)>\ve^{-(1/2+1/2^{n+1})(k-(M_n+K_n))+\rho(M_n+K_n)} 
$$
Now we use the fact that $k>t_0+K_n>M_n^{3/2}+K_n>M_n^{3/2}$ to conclude (by using the same type of argument as in the proof
of Sublemma \ref{ind_sl2}) that right hand side of this last expression is larger than
$\ve^{-(1/2+1/2^{n+2})k}$.
We have thus shown that $(3)_{n+1}[t_1]$ holds. 

We continue by showing that $(3)_{n+1}[t_2]$ holds. Assume that $(x_k,y_k)\in (\T\setminus G_n)\times A'$ for some $t_1+1\leq k \leq t_2$ and some $y$. We note that we must have
$k>t_1+K_n$.
By the same argument as above we can conclude that $y_{t_i+K_n}, y_{t_i-M_n}\in A$ ($i=1,2$), for otherwise we cannot have $y_k\in A'$.
We now split the product $\prod_{j=0}^{k-1} f'(y_j)$ into five parts: 
$\prod_{j=0}^{t_0-M_n-1}\prod_{j=t_0-M_n}^{t_0+K_n-1}\prod_{j=t_0+K_n}^{t_1-M_n-1}$ $\prod_{j=t_1-M_n}^{t_1+K_n-1}\prod_{j=t_1+K_n}^k$. Applying $(3)_n$ to the points $(x,y), (x_{t_i+K_n},y_{t_i+K_n})$ (i=1,2),
and using the global lower bound on $f'$, we get (as above) the estimate
$$
\prod_{j=0}^{k-1} f'(y_j)>\ve^{-(1/2+1/2^{n+1})(k-2(M_n+K_n))+2\rho(M_n+K_n)}> \ve^{-(1/2+1/2^{n+2})k}
$$
since $k>t_1+K_n>t_0+M_n^2+K_n\gg 2M_n^{3/2}$. Thus $(3)_{n+1}[t_2]$ also holds.

Proceeding inductively (considering the time intervals $[t_2,t_3],\ldots, [t_{s-1},t_s]$) 
shows that $(3)_{n+1}[t_s]$ holds. 
\end{proof}

\noindent\emph{Step 4}(Verifying $(4)_{n+1}$). Now we shall control the geometry of the iterates of the segment $\Gamma_{n+1}=(J_{n+1}-M_{n+1}\omega)\times \{\beta\}$. 
As in the base case we get two cases, depending on whether the intervals  
$I_{n+1}$ were resonant or not (the analysis in step 1 above). We begin with the resonant case (R) where $J_{n+1}=I_{n+1}^1\cup (I_{n+1}^2-\nu_{n+1}\omega)$ is a single interval.
This is the more involved case since we need to use "shadowing" twice. 

\smallskip
Case R: To verify $(4)_{n+1}$ we need to control $T^{M_{n+1}+K_{n+1}}(\Gamma_{n+1})$. We will take the following route: control
$T^{M_{n+1}-M_n}(\Gamma_{n+1})$, use $(4)_n$ to control $T^{M_{n+1}+K_n}(\Gamma_{n+1})$, control $T^{M_{n+1}+\nu_{n+1}-K_n}(\Gamma_{n+1})$, 
again use $(4)_n$ to control $T^{M_{n+1}+\nu_{n+1}+M_n}(\Gamma_{n+1})$, and finally
control $T^{M_{n+1}+K_{n+1}}(\Gamma_{n+1})$.

For convenience we shall collect some of the set inclusions which will be used along the argument. First we recall (\ref{ind_eq92}), that
$J_{n+1},J_{n+1}+\nu_{n+1}\omega\subset J_n$. Thus, by (\ref{ind_eq8}) we have
$$
J_{n+1}+K_n\omega,J_{n+1}+(\nu_{n+1}+K_n)\omega\subset X_n'. 
$$
Furthermore, by (\ref{ind_eq8p}) (recall the definitions of $G_n$ and $H_n$) we have 
$$
J_{n+1}\subset H_n \text{ and } J_{n+1}-M_n\omega,J_{n+1}+(\nu_{n+1}-M_n)\omega\subset \T\setminus G_n. 
$$
We also recall that $(\ref{ind_eq4})_{n+1}$ holds, and conclude that
$$
J_{n+1}-M_{n+1}\omega\subset X_n, \quad J_{n+1}+K_{n+1}\omega\subset H_n\subset \T\setminus G_n.
$$
Finally, since we are in the resonant case (R) we note that, by (\ref{ind_eq90}), 
\begin{equation}\label{ind_eq90p}
J_{n+1}\cap I_{n+1}=I_{n+1}^1, \quad (J_{n+1}+\nu_{n+1})\omega\cap I_{n+1}=I_{n+1}^2.
\end{equation}

We now turn to the analysis. We recall Lemma \ref{P_L0} and Lemma \ref{SF_L3}(4), i.e., $T$ preserves orientation and increasing functions are mapped to strictly increasing functions. 
Moreover, we also recall that if $\psi_0$ is a strictly increasing function which intersects the constant function $\eta_0=\beta$ at $x=\xi$, so trivially $\psi$ overtakes $\eta_0$ at $x=\xi$, 
then, by Lemma \ref{Tf_0}, the same holds for the iterates, i.e.,  $\psi_k$ overtakes $\eta_k$ at $x=\xi$ for all $k$.
These
facts will be used several times below, without explicitly writing it out. 

Let 
$$
y_k(x)=\pi_2(T^k(x-M_{n+1}\omega,\beta)).
$$
Then we can write 
$$
T^{k}(\Gamma_{n+1})=\{(x+(k-M_{n+1})\omega,y_{k}(x)):x\in J_{n+1})\}
$$
In particular we have
$T^{M_{n+1}+K_{n+1}}(\Gamma_{n+1})=\{(x+K_{n+1}\omega,y_{M_{n+1}+K_{n+1}}(x)):x\in J_{n+1})\}$. Thus, we need to control $y_{M_{n+1}+K_{n+1}}$ for $x\in J_{n+1}$.
We shall write $J_{n+1}=[s,t]$.

We first derive an upper bound on $y_{M_{n+1}+K_{n+1}}'(x)$. Since $y_0=\beta\notin A$ and $J_{n+1}-M_{n+1}\omega\subset X_n$ we can apply Sublemma \ref{ind_sl4} and derive (since  
$J_{n+1}\subset H_n$ and $N(x;I_{n+1})\geq M_{n+1}$ for all $x\in J_{n+1}-M_{n+1}\omega$)
$$
\prod_{j=k}^{M_{n+1}} f'(y_j(x))<\ve^{(M_{n+1}-k+1)/2} \text{ on } J_{n+1},  1\leq k \leq M_{n+1}.
$$
Below we shall use the fact that Sublemma \ref{ind_sl4} also gives, since $J_{n+1}-M_n\omega\subset \T\setminus G_n$,
\begin{equation}\label{step4_eq1}
y_{M_{n+1}-M_n}(x)\in \T\setminus A' \text{ for all } x\in J_{n+1}.
\end{equation}
Applying Lemma \ref{SF_L3}(1) gives us $y'_{M_{n+1}+1}(x)<2\kappa$ on $J_{n+1}$. Hence  (trivially)
\begin{equation}\label{step4_eq0}
y'_{M_{n+1}+k}(x)<3\kappa \ve^{-\rho (k-1)}< \ve^{-\rho k}\text{ on } J_{n+1} \text{ for any } k\geq 1
\end{equation}
by Lemma \ref{SF_L3}(3).

From (\ref{step4_eq1}) we have some control on $y_{M_{n+1}-M_n}$ on $J_{n+1}$. Next we want to control $y_{M_{n+1}+K_n}$. To do this we use shadowing, i.e., we shall compare the
iterates of the points $(x-M_n\omega,y_{M_{n+1}-M_n}(x))$ by those of $(x-M_n\omega,\beta)$, for $x\in J_{n+1}$. Let 
$$
\eta_k(x)=\pi_2(T^k(x-M_{n}\omega,\beta)).
$$  
From $(4)_n$ we have $\eta_{M_n+K_n}(I^1_{n+1})=A'$ and, since (\ref{ind_eq90p}) holds, we have $\eta_{M_n+K_n}(x)\notin A'$ for all $x\in J_{n+1}\setminus I_{n+1}^1$. Applying Sublemma
\ref{ind_sl2} (recall that (\ref{step4_eq1}) holds) give us that $y_{M_{n+1}+K_n}\approx \eta_{M_{n}+K_n}$ on $J_{n+1}$. Thus we can use Lemma \ref{Tf_1} and 
conclude that there is a unique point $\xi$ in $J_{n+1}=[s,t]$, lying in the 
interior of $J_{n+1}$, such that 
\begin{equation}\label{step4_eq2p}
y_{M_{n+1}+K_n}(\xi)=\beta,
\end{equation}
and 
\begin{equation}\label{step4_eq2}
[\beta,y_{M_{n+1}+K_n}(s)],[\beta,y_{M_{n+1}+K_n}(t)]\subset \T\setminus A
\end{equation}
(recall the definition of $A''\supset A$ in (\ref{app})). 

Now we shall get a control on $y_{M_{n+1}+(\nu_{n+1}-M_{n})}$. Recall from (\ref{ind_eq91}) that we have $\nu_{n+1}\geq M_{n}^2$. Let 
$$
\eta_k(x)=\pi_2(T^k(x+K_{n}\omega,\beta))
$$  
(we "recycle" the variables $\eta_k$ in order not to have too many variables). Since $J_{n+1}+K_n\omega\subset X_n$ and
$J_{n+1}+(\nu_{n+1}-M_n)\omega\subset \T\setminus G_n$, and since $N(x;I_{n+1})\geq (\nu_{n+1}-K_n)$ for all $x\in J_{n+1}+K_n\omega$, it follows
from Sublemma \ref{ind_sl4} that $\eta_{\nu_{n+1}-K_n-M_n}(x)\in \T\setminus A'$ for all $x\in J_{n+1}$. Moreover, since (\ref{step4_eq2}) holds
the same argument also shows that 
$$[\eta_{\nu_{n+1}-K_n-M_n}(x), y_{M_{n+1}+\nu_{n+1}-M_n}(x)]\subset  \T\setminus A' \text{ for } x=s,t.
$$
This last inclusion trivially implies (see Lemma \ref{P_betalemma}) that
\begin{equation}\label{step4_eq3}
[\beta,y_{M_{n+1}+\nu_{n+1}-M_n}(s)],[\beta,y_{M_{n+1}+\nu_{n+1}-M_n}(t)]\subset \T\setminus A.
\end{equation}
Since (\ref{step4_eq2p}) holds we also note that $y_{M_{n+1}+\nu_{n+1}-M_n}(x)$ overtakes $\eta_{\nu_{n+1}-K_n-M_n}(x)$ at $x=\xi$ 
(recall Lemma \ref{Tf_0}). Applying Lemma \ref{Tf_2}(1), with $I=J_{n+1}$, 
shows that there is a unique point $\widetilde{\xi}\in J_{n+1}$ such that $y_{M_{n+1}+\nu_{n+1}-M_n}(\widetilde{\xi})=\beta$.

To estimate $y_{M_{n+1}+\nu_{n+1}+K_n}$ we use shadowing again. More precisely,  we shall compare the iterates of 
$(x+(\nu_{n+1}-M_n)\omega,y_{M_{n+1}+\nu_{n+1}-M_n}(x))$ by the iterates of $(x+(\nu_{n+1}-M_n)\omega,\beta)$. Therefore we let (again recycling the variables $\eta_k$)
$$
\eta_k(x)=\pi_2(T^k(x+(\nu_{n+1}-M_{n})\omega,\beta)).
$$
We recall that $J_{n+1}+(\nu_{n+1}-M_n)\omega\subset J_n-M_n\omega$ and $(J_{n+1}+\nu_{n+1}\omega)\cap I_{n+1}=I_{n+1}^2$. Thus it follows from $(4)_n$
that $\eta_{M_n+K_n}(I^2_{n+1}-\nu_{n+1}\omega)=A'$ and we have $\eta_{M_n+K_n}(x)\notin A'$ for all $x\in J_{n+1}\setminus (I_{n+1}^2-\nu_{n+1}\omega)$. 
In particular we have $\eta_{M_n+K_n}(x)\notin \inn(A')$ for $x=s,t$.
We also note that, by the definition of $\widetilde{\xi}$ above,
$y_{M_{n+1}+\nu_{n+1}+K_n}$ overtakes $\eta_{M_n+K_n}$ at $x=\widetilde{\xi}$. Moreover, using (\ref{step4_eq3}) and Sublemma \ref{ind_sl2} we conclude that
$$[\eta_{M_n+K_n}(x),y_{M_{n+1}+\nu_{n+1}+K_n}(x)]\subset \T\setminus A'' \text{ for } x=s,t.$$ We are thus in a position where we 
can apply Lemma \ref{Tf_3} and conclude that there are exactly two $\xi_1,\xi_2\in J_{n+1}$ such that $y_{M_{n+1}+\nu_{n+1}+K_n}(\xi_i)=\beta$, both lying in the interior of $J_{n+1}$, and
\begin{equation}\label{step4_eq4}
[\beta,y_{M_{n+1}+\nu_{n+1}+K_n}(x)]\subset \T\setminus A \text{ for } x=s,t.
\end{equation}
To show that $I_{n+2}$, as defined in $(4)_{n+1}$, is "well inside" $J_{n+1}$ we shall also use the following easy observation.
From (\ref{step4_eq0}) we know that $y_{M_{n+1}+\nu_{n+1}+K_n}'(x)< \ve^{-\rho(\nu_{n+1}+K_n)}$ on $J_{n+1}$. It thus follows from Lemma \ref{SF_L4} that
$$
\dist(\widetilde{I}_{n+2},\partial J_{n+1})>\ve^{\rho(\nu_{n+1}+K_n)+1}/2\gg \ve^{K_{n+1}^{2/3}}
$$
where $\widetilde{I}_{n+2}=\{x\in J_{n+1}:y_{M_{n+1}+\nu_{n+1}+K_n}(x)\in A\}$. In the last inequality we used (\ref{ind_eq91}).

The last step is to control $y_{M_{n+1}+K_{n+1}}$. Let 
$$
\eta_k(x)=\pi_2(T^k(x+(\nu_{n+1}+K_{n})\omega,\beta)).
$$ 
As we recalled above we have $J_{n+1}+(\nu_{n+1}+K_n)\omega\subset X_n$ and $J_{n+1}+K_{n+1}\omega\subset \T\setminus G_n$, and by (\ref{ind_eq90}) we have 
$N(x;I_{n+1})\gg K_{n+1}$ for all $x\in J_{n+1}+(\nu_{n+1}+K_{n})\omega$. Therefore it follows from Sublemma \ref{ind_sl4} that
$$\eta_{K_{n+1}-\nu_{n+1}-K_n}(x)\notin A' \text{ for all } x\in J_{n+1};$$ and by the definitions of $\widetilde{I}_{n+2}$ above we get that $y_{M_{n+1}+K_{n+1}}(x)\notin A'$ for all 
$x\in J_{n+1}\setminus \widetilde{I}_{n+2}$. Hence 
$$I_{n+2}\subset \widetilde{I}_{n+2}.$$ We also get,
since (\ref{step4_eq4}) holds, $$[\eta_{K_{n+1}-\nu_{n+1}-K_n}(x),y_{M_{n+1}+K_{n+1}}(x)]\subset \T\setminus A' \text{ for } x=s,t.$$
We can therefore apply Lemma \ref{Tf_2}(2) (recall that $y_{M_{n+1}+K_{n+1}}$ overtakes $\eta_{K_{n+1}-\nu_{n+1}-K_n}$ at $x=\xi_1,\xi_2$) and conclude that
$\{x\in J_{n+1}: y_{M_{n+1}+K_{n+1}}(x)\in A'\}$ consists of two intervals, $I_{n+2}^i (i=1,2)$, and $y_{M_{n+1}+K_{n+1}}(I_{n+2}^i)=A'$ ($i=1,2$). 

It remains to 
get a lower bound on the derivative of $y_{M_{n+1}+K_{n+1}}$ on $I_{n+2}$. To obtain this we use Sublemma \ref{ind_sl5} for the passage from
$J_{n+1}+(\nu_{n+1}+K_n)\omega$  to $J_{n+1}+K_{n+1}\omega$. 
We conclude that if $y_{M_{n+1}+K_{n+1}}(x)\in A'$ for some $x\in J_{n+1}+(\nu_{n+1}+K_n)\omega\subset J_{n}+K_{n}\omega\subset X'_n$, then 
$$
\prod_{k=\nu_{n+1}+K_n}^{K_{n+1}-1}f'(y_{M_{n+1}+k}(x))>\ve^{-(1/2+1/2^{n+2})(K_{n+1}-(\nu_{n+1}+K_n))}\gg\kappa\ve^{-K_{n+1}/2} 
$$
whenever the derivatives exist. The last inequality follows from the fact that (\ref{ind_eq91}) holds, i.e., $\nu_{n+1}^2\leq K_{n+1}$, and we have $K_{n+1}^{1/2}/2^{n+2}\gg 1$ (and $K_n\ll K_{n+1}$).
By applying Lemma \ref{SF_L3}(2), recalling the (global) trivial lower bound $y_{M_{n+1}+\nu_{n+1}+K_n}'(x)>1/\kappa$ given by Lemma \ref{SF_L3}(4), we obtain the desired lower bound on
$y_{M_{n+1}+K_{n+1}}'$ on $I_{n+2}$.

\smallskip
Case NR: In this case $J_{n+1}=I_{n+1}$ consist of two intervals, and we only need to use shadowing once (since there is no interactions). We can treat these two intervals
independently, like we did in the proof of Lemma \ref{FS_L2}. We focus on one of them, say $I_{n+1}^1=[s,t]$. The route we shall take is the following:  control
$T^{M_{n+1}-M_n}(\Gamma_{n+1}^1)$, use $(4)_n$ to control $T^{M_{n+1}+K_n}(\Gamma_{n+1}^1)$, and control $T^{M_{n+1}+K_{n+1}}(\Gamma_{n+1}^1)$.

As before we let 
$$
y_k(x)=\pi_2(T^k(x-M_{n+1}\omega,\beta)),
$$
and we need estimates of $y_{M_{n+1}+K_{n+1}}$. By proceeding exactly as above we get the following estimates on $y_{M_{n+1}+K_n}$: there is a unique point $\xi\in I_{n+1}^1=[s,t]$, 
lying in the interior of $I_{n+1}^1$, such that $y_{M_{n+1}+K_n}(\xi)=\beta$, and 
$[\beta,y_{M_{n+1}+K_n}(s)],[\beta,y_{M_{n+1}+K_n}(t)]\subset \T\setminus A$. We let
$$
\eta_k(x)=\pi_2(T^k(x+K_{n}\omega,\beta)).
$$ 
Since $I_{n+1}^1+K_n\omega\subset X_n$ and $N(x; I_{n+1})\gg K_{n+1}-K_n$ for all $x\in I_{n+1}+K_n\omega$ we can apply Sublemma \ref{ind_sl4} to conclude that 
$\eta_{K_{n+1}-K_n}(x)\notin A'$ for all $x\in I_{n+1}^1$. By applying Lemma \ref{Tf_2}(1) we  see that 
$\{x\in I_{n+1}^1: y_{M_{n+1}+K_{n+1}}(x)\in A'\}$ consists of one interval, $I_{n+2}^1$, and $y_{M_{n+1}+K_{n+1}}(I_{n+2}^1)=A'$. 
It also follows that $I_{n+2}^1\subset \{x\in I_{n+1}^1:y_{M_{n+1}+K_n}(x)\in A\}$ which is "well" inside $I_{n+1}^1$, as in the previous case.
Proceeding as in the resonant case above
we obtain the desired lower bound on the derivative of $y_{M_{n+1}+K_{n+1}}$ on $I_{n+2}^1$.

\medskip
This finishes the proof of the inductive lemma.
\end{proof}

\end{section}

\begin{section}{Miscellaneous}
In this section we have collected results of a computational nature. The first three lemmas are related to number theory; the last one
is from ergodic theory.

The first lemma gives a lower bound on the return-time to an interval under translation by a Diophantine number. 
\begin{lemma}\label{M_L1}
Assume that $\omega$ satisfies the Diophantine condition (\ref{DC}) for some $\gamma>0$. If $I\subset \T$ is an interval of length 
$\leq \delta$, then 
$$
I\cap(I+k\omega)=\emptyset
$$
for all $0<|k|\leq \sqrt{\gamma/\delta}$.
\end{lemma}
\begin{proof}If $I\cap(I+k\omega)\neq \emptyset$ we must have $|k\omega-p|\leq \delta$ for some integer $p$. Thus, using (\ref{DC}) we get that either
$\gamma/k^2<\delta$, i.e., $|k|>\sqrt{\gamma/\delta}$, or $k=0$.
\end{proof}

\begin{lemma}\label{M_L2}
Assume that $\omega$ satisfies the Diophantine condition (\ref{DC}) for some $\gamma>0$, and assume that $I^1,I^2\subset \T$ 
are two disjoint intervals of length $\leq \delta$. Let $N=[\sqrt{\gamma/(2\delta)}]$ and assume that there is an integer $\nu$, $0<\nu<N$, such that
$I^1\cap (I^2-\nu\omega)\neq \emptyset$. If we let $J$ be the interval $J=I^1\cup (I^2-\nu\omega)$, then
$$
(J+k\omega)\cap I=\emptyset \text{ for all } k\in[\nu-N,N]\setminus \{0,\nu\},
$$
where $I=I^1\cup I^2$. Moreover, $J\cap I=I^1$ and $(J+\nu\omega)\cap I=I^2$.
\end{lemma}
\begin{proof}Since $|J|\leq 2\delta$, it follows from Lemma \ref{M_L1}, and the choice of $N$, that $(J+k\omega)\cap J=\emptyset$ for all $0<|k|\leq N$.
Moreover,  $(J+k\omega)\cap (J+\nu\omega)\neq \emptyset$ iff $(J+(k-\nu)\omega)\cap J=\emptyset$. Thus, by the same argument we have
$(J+k\omega)\cap (J+\nu\omega)= \emptyset$ for all $0<|k-\nu|\leq N$. From this is now follows
(since $I^1\subset J$ and $I^2\subset J+\nu\omega$) that $(J+k\omega)\cap I=\emptyset$ for all
$k\in[\nu-N,N]\setminus \{0,\nu\}$ (recall that $0<\nu<N$). This shows the first part of the lemma-


To finish we first note that (again by Lemma \ref{M_L1}) $(I^i+k\omega)\cap I^i=\emptyset$ for all $0<|k|\leq N$ ($i=1,2$).
Thus, we must have $J\cap I=(I_1\cup (I_2-\nu\omega))\cap (I_1\cup I_2)=I_1$ and $(J+\nu\omega)\cap I=((I_1+\nu\omega)\cup I_2)\cap (I_1\cup I_2)=I_2$. 
\end{proof}

The following lemma will be used in the inductive step in the proof of Proposition \ref{pt_prop} to show that there is always "space"
for a small set consisting of tiny intervals to be translated into a "good" position.

\begin{lemma}\label{M_L3}
Assume that each $I_j\subset \T$ ($j=0,1,\ldots,n$) consists of at most $p\geq 1$ intervals. 
Assume further that there are integers $N_0<N_1<\cdots <N_n$ such that for each interval $I_j^\ell$ in $I_j$ ($j\in [0,n]$) we have
\begin{equation}\label{misc_L1}
3I_j^\ell\cap\bigcup_{m=1}^{N_j}(3I_j^\ell+m\omega)=\emptyset.
\end{equation}
If the set $J\subset \T$ consists of at most $p$ intervals, $\max_i|J^i|< \min_{j,\ell}|I_j^\ell|$, 
and if the integers $Z_j>0$ are chosen so that $10p^2Z_j/N_j<3^{-(j+1)}$, then there is an integer $k\in [1,N_n]$ (in fact at least $N_n/2$ of them) such that
\begin{equation}\label{misc_L2}
(J+k\omega)\subset \T\setminus \bigcup_{j=0}^n\bigcup_{m=-Z_j}^{Z_j}(I_j+m\omega).
\end{equation}
\end{lemma}
\begin{proof}
We first fix $j\in [0,n]$ and pick one interval $J^i$ from $J$, and one from $I_j$, denoted $I_j^\ell$. 
From condition (\ref{misc_L1}), and the fact that $|J^i|< |I_j^\ell|$ (so if $(J^i+k\omega)\cap I_j^\ell\neq \emptyset$, then $(J^i+k\omega)\subset 3I_j^\ell$), 
we deduce that in the interval $[1,N_j]$ there are at most $2Z_j+1$ integers $k$ such that 
$$
(J^i+k\omega)\cap\bigcup_{m=-Z_j}^{Z_j}(I^\ell_j+m\omega) \neq \emptyset.
$$
Therefore, in $[1,N_n]$ there are at most $([N_n/N_j]+1)(2Z_j+1)< N_n(10Z_j/N_j)$ integers $k$ such that the above condition holds. 
Since $J$ and $I_j$ each consists of at most $p$ intervals, there are at most $p^2N_n(10Z_j/N_j)$ integers $k\in[1,N_n]$ such that
$$
(J+k\omega)\cap\bigcup_{m=-Z_j}^{Z_j}(I_j+m\omega) \neq \emptyset.
$$
Consequently, 
using the estimates on $10Z_j/N_j$ we conclude that there are at most
$$
p^2N_n\left( 10Z_0/N_0+10Z_1/N_1+\ldots+10Z_n/N_n \right)<N_n\sum_{j=0}^n\frac{1}{3^{j+1}}<N_n/2
$$ 
integers $k$ in $[1,N_n]$ for which (\ref{misc_L2}) is violated. Thus, the statement  of the lemma holds.
\end{proof}

The next lemma is an elementary topological fact (which we use in the proof of Proportion \ref{pt_prop}).
\begin{lemma}\label{M_L4}
Assume that $\emptyset\neq J_n\subset S_n\subset \T$ ($n\geq 0$) are closed sets such that 
 $S_{n+1}\subset S_n$ for all $n\geq 0$. 
If $x^*\in \T$ is such that $\dist(x^*,J_n)\to 0$ as $n\to\infty$, then $x^*\in \cap_{n\geq 0} S_n$.
\end{lemma}
\begin{proof}We prove this by contradiction. Assume that $x^*\notin  \cap_{n\geq 0} S_n$. Then there is an $m\geq 0$
such that $x^*\notin S_m$, and thus $x^*\notin S_n$ for all $n\geq m$. Since $S_m$ is closed we must have $\dist(x^*,S_m)>\delta$ for some
$\delta>0$. But this implies $\dist(x^*,S_n)>\delta$ for all $n\geq m$, and hence $\dist(x^*,J_n)>\delta$ for all $n\geq m$. 
\end{proof}

The last proposition is  due to Herman \cite{H1,H2} and Johnson \cite{Jo}. However, the formulation is slightly different, so for completeness we include a proof. 
\begin{prop}\label{M_P}
Assume that $\omega\in \T$ is irrational and that the map $F:\T^2\to\T^2$ 
is of the form $F(x,y)=(x+\omega,h(x,y))$, where $h:\T^2\to\T$ is continuous. Assume further that there are two
measurable functions $w^\pm:\T\to\T$ such that $w^+(x)\neq w^-(x)$ for a.e. $x\in\T$, 
\begin{equation}\label{M_P_eq1}
F(x,w^\pm(x))=(x+\omega,w^\pm(x+\omega)) \text{ a.e. } x\in\T,
\end{equation}
and such that for a.e. $x\in \T$ there holds
\begin{equation}\label{M_P_eq2}
d(y_n,w^+(x_n))\to 0 \text{ as } n\to\infty \text{ for all } y\neq w^-(x).
\end{equation}
Then $F$ has exactly two invariant ergodic Borel probability measures $\mu^{\pm}$; and $\mu^\pm$
is the push-forward of the Lebesgue measure on $\T$ by the map $x\mapsto (x,w^\pm(x))$.
Moreover, $F$ has either only one or only two minimal sets; either $M=\text{supp }\mu^+=\text{supp }\mu^-$, or $M^\pm=\text{supp }\mu^\pm$.    
\end{prop} 
\begin{proof}
Let $\mu^\pm$ be the push-forward of the Lebesgue measure on $\T$ (which is unique probability measure on $\T$ invariant under translation $x\mapsto x+\omega$) 
by the maps $x\mapsto (x,w^\pm(x))$. From (\ref{M_P_eq1}) it easily follows that $\mu^\pm$ are $F$-invariant ergodic Borel probability measures on $\T^2$.

We shall now show that these are the only such measures. First we note that by Birkhoff's ergodic theorem the following holds for all $\vf\in C^0(\T^2)$:
\begin{equation}\label{M_P_eq3}
\lim_{n\to\infty}\frac{1}{n}\sum_{k=0}^{n-1}\varphi(F^k(x,w^\pm(x)))=\int_{\T^2}\varphi d\mu^\pm, \text{ for a.e. }  x\in\T.
\end{equation}
Using (\ref{M_P_eq2}) we note that for a.e. $x\in \T$ and all $y\neq w^-(x)$ we have
\begin{equation}\label{M_P_eq4}
\lim_{n\to\infty}\frac{1}{n}\left(\sum_{k=0}^{n-1}\varphi(F^k(x,w^+(x)))-\sum_{k=0}^{n-1}\varphi(F^k(x,y))\right)=0
\end{equation}
for all $\vf\in C^0(\T^2)$. Thus, for a.e $x\in \T$ we have 
\begin{equation}\label{M_P_eq5}
\sum_{k=0}^{n-1}\varphi(F^k(x,y))=\begin{cases}\int_{\T^2}\varphi d\mu^+, &\text{ if } y\neq w^-(x) \\ \int_{\T^2}\varphi d\mu^-, 
&\text{ if } y= w^-(x). \end{cases}
\end{equation}

Assume that $\mu$ is an $F$-invariant ergodic Borel probability measure. Assume also that $\mu\neq \mu^-$. We shall show that then $\mu=\mu^+$.
Since $\mu$ and $\mu^-$ are both ergodic they must
be mutually singular. Thus we can find a set $A\subset \T^2$ of full $\mu$-measure such that $\mu^-(A)=0$. 
The projection of $\mu$ onto the base,
i.e., $\lambda(I)=\mu(\pi_1^{-1}(I))$, is a probability measure on $\T$ which is invariant under $x\mapsto x+\omega$. Hence it must be the Lebesgue measure.  
It thus follows that $|\pi_1(A)|=1$. Since $\mu^-(A)=0$ we must have $(x,w^-(x))\notin A$ for a.e. $x\in \T$.
It therefor follows from Birkhoff's ergodic theorem and (\ref{M_P_eq5}) that we for all $\vf\in C^0(\T^2)$ have
$$
\int_{\T^2}\varphi d\mu=\lim_{n\to\infty}\frac{1}{n}\sum_{k=0}^{n-1}\varphi(F^k(x,y))=\int_{\T^2}\varphi d\mu^+ \text{ for $\mu$-a.e } (x,y)\in A. 
$$ 
We thus conclude that $\int_{\T^2}\vf d\mu=\int_{\T^2}\vf d\mu^+$ for all $\vf\in C^0(\T^2)$. Therefore, by Riesz representation theorem, we have $\mu=\mu^+$. 

We now turn to the minimal sets. Firstly, since $F$ is continuous and $\T^2$ is compact, we know that $F$ has a minimal set $M$ (in particular we have $F(M)=M$). 
Consider the restriction $F|_M:M\to M$. Since $M$ is compact, this map has an $F|_M$-invariant ergodic Borel probability measure $\widetilde{\mu}$ on $M$.
Extending this measure to $\T^2$ (by letting $\mu(A)=\widetilde{\mu}(A\cap M)$) gives us an $F$-invariant ergodic Borel probability measure $\mu$.
Hence we must have $\mu=\mu^+$ or $\mu=\mu^-$. Since the measure $\mu$ is supported on $M$, and since the support is a closed and $F$-invariant set,
and since $M$ is minimal, we must have $\text{supp } \mu=M$.

From this we conclude that $F$ can have at most two minimal sets, since each minimal set supports an $F$-invariant ergodic Borel probability measure.
We also note that the support of an $F$-invariant measure is closed and $F$-invariant; hence it contains a minimal set.
Thus, either $F$ has a unique minimal set $M$, and $M=\text{supp }\mu^+=\text{supp }\mu^-$, or $F$ has two minimal sets, $M^\pm$, and 
$M^\pm=\text{supp }\mu^\pm$. 

\end{proof}

\end{section}

\end{document}